\documentclass{amsart}
\usepackage{amssymb}
\usepackage{graphicx}
\usepackage{epstopdf}
\usepackage{pdfsync}

\newcommand{\hs}{\hspace{0.15cm}}

\numberwithin{equation}{section}

\theoremstyle{plain} % This is the default.
\newtheorem{thm}[equation]{Theorem}

\newtheorem{lem}[equation]{Lemma}

\newtheorem{claim}[equation]{Claim}

\theoremstyle{definition}
\newtheorem{defn}[equation]{Definition}

\theoremstyle{remark}
\newtheorem{rem}[equation]{Remark}

\graphicspath{{PSF_figures/}}
\title[Primitive, proper power, and Seifert curves] {Primitive, proper power, and Seifert curves in the boundary of a genus two handlebody}

\author{Sungmo Kang}

\email{skang4450@chonnam.ac.kr}
\begin{document}


%	\thanks{}

%    General info
%\subjclass[2000]{Primary 57M50}

%\date{}

%\dedicatory{}

%\keywords{Dehn surgery, lens space, double-primitive, 3-sphere}

\begin{abstract}
A simple closed curve $\alpha$ in the boundary of a genus two handlebody $H$ is \textit{primitive} if adding a 2-handle to $H$ along $\alpha$
yields a solid torus. If adding a 2-handle to $H$ along $\alpha$ yields a Seifert-fibered space and not a solid torus, the curve is called Seifert.
If $\alpha$ is disjoint from an essential separating disk in $H$, does not bound a disk in $H$, and is not primitive in $H$, then it is said to be \textit{proper power}.

As one of the background papers of the classification project of hyperbolic primitive/Seifert knots in $S^3$ whose complete list is given in \cite{BK20},
this paper classifies in terms of R-R diagrams primitive, proper power, and Seifert curves. In other words, we provide up to equivalence all possible R-R diagrams
of such curves. Furthermore, we further classify all possible R-R diagrams of proper power curves with respect to an arbitrary complete set of cutting disks of a genus two handlebody.
\end{abstract}

\maketitle

%\tableofcontents

\section{Introduction and main results}\label{Introduction and main result}

In this paper, we provide the classifications of three types of simple closed curves lying in the boundary of a genus two handlebody: primitive, proper power, and Seifert curves. These classifications will be used in the classification project of hyperbolic primitive/Seifert knots in $S^3$ whose complete list is given in \cite{BK20}.

Primitive/Seifert(or simply P/SF) knots, which were introduced in \cite{D03}, are a natural generalization of primtive/primitive(or simply P/P) knots defined by Berge in \cite{B90} or an available version \cite{B18}. Both P/P knots and P/SF knots are represented by simple closed curves lying in a genus two Heegaard surface of $S^3$ bounding two handlebodies such that 2-handle additions to the handlebodies along the curves are solid tori for P/P knots and one 2-handle addition is a solid torus and the other is a Seifert-fibered space and not a solid torus for P/SF knots. One component of the intersection of a regular neighborhood of a knot and the Heegaard surface defines a so-called surface-slope.

Berge constructed twelve families of P/P knots which are referred to as the Berge knots. The Berge knots admit lens space Dehn surgeries at surface-slopes. The Berge conjecture, which is still unsolved, says that if a knot in $S^3$ admits a lens space Dehn surgery, then it is a Berge knot and the surgery is the corresponding surface-slope surgery. Therefore the conjecture implies the complete classification of knots admitting lens space Dehn surgeries. Toward the Berge conjecture, it is proved in \cite{B08} or independently in \cite{G13} that all P/P knots are the Berge knots. This implies that the Berge knots are the complete list of P/P knots.

Meanwhile, P/SF knot are also of interest, because P/SF knots admit Seifert-fibered Dehn surgeries at surface-surface slopes and knots with Dehn surgeries yielding Seifert-fibered spaces are not well understood. The classification project of hyperbolic primitive/Seifert knots in $S^3$ has been carried out for years and has recently been completed. The complete list of hyperbolic primitive/Seifert knots in $S^3$ is given in \cite{BK20} where the surface-slope of the exceptional surgery on each P/SF knot that yields a Seifert-fibered space and the indexes of each exceptional fiber in the resulting Seifert-fibered space are also provided.

Now we describe the results of this paper, which is the classifications of primitive, proper power, and Seifert curves in the boundary of a genus two handlebody. The definitions of such curves are as follows.

\begin{defn}
Let $H$ be a genus two handlebody, $\alpha$ an essential simple closed
curve in $\partial H$, and $H[\alpha]$ the 3-manifold obtained by
adding a 2-handle to $H$ along $\alpha$.
\begin{enumerate}
\item $\alpha$ is said to be \textit{primitive} if $H[\alpha]$ is a solid torus.
\item $\alpha$ is said to be \textit{proper power} if $\alpha$ is disjoint from an essential separating disk in $H$, does not bound a disk in $H$, and is not primitive in $H$.
\item $\alpha$ is said to be \textit{Seifert} if $H[\alpha]$ is a Seifert-fibered space and not a solid torus.
\end{enumerate}
\end{defn}

There are subtypes of Seifert curves in $H$. Since $H$ is a genus two handlebody, that $\alpha$ is Seifert in $H$ implies that
$H[\alpha]$ is an orientable Seifert-fibered space over $D^2$
with two exceptional fibers, or an orientable Seifert-fibered space
over the M\"{o}bius band with at most one exceptional fiber. Therefore, we further
divide Seifert curves into two subtypes. If
$H[\alpha]$ is Seifert-fibered over $D^2$, we say that
$\alpha$ is \textit{Seifert-d}. If $H[\alpha]$ is Seifert-fibered over the M\"{o}bius band, we say that
$\alpha$ is \textit{Seifert-m}.

The following theorems present the classifications of primitive, proper power, and Seifert curves. They are described in terms of R-R diagrams. For the definition and properties of R-R diagrams, see \cite{K20}.

\begin{thm}\label{main theorem}
Suppose $\alpha$ is a simple closed curve in the boundary of a genus two handlebody $H$.
\begin{enumerate}
\item If $\alpha$ is a primitive curve, then $\alpha$ has an R-R diagram of the form shown in Figure~\emph{\ref{PSFFig1aa1}}.
\item If $\alpha$ is a Seifert-d curve, then $\alpha$ has an R-R diagram of the form shown in Figure~\emph{\ref{PSFFig2a1}} with $n, s > 1$, $a, b > 0$, and $\gcd(a,b) = 1$.
\item If $\alpha$ is a Seifert-m curve, then $\alpha$ has an R-R diagram of the form shown in Figure~\emph{\ref{SF_on_Mobius11}} with $s>1$.
\item If $\alpha$ is a proper power curve, then $\alpha$ has an R-R diagram of the form shown in Figure~\emph{\ref{PPower8}} with $s>1$.
\end{enumerate}

Regarding Seifert curves, if $\alpha$ has an R-R diagram of the form shown in Figure~\emph{\ref{PSFFig2a1}}\emph{a} \emph{(}\emph{b}, respectively\emph{)}, then $H[\alpha]$ is a Seifert-fibered space over $D^2$ with two exceptional fibers of indexes $n$ and $s$ \emph{(}$n(a+b)+b$ and $s$, respectively\emph{)}. If $\alpha$ has an R-R diagram of the form shown in Figure~\emph{\ref{SF_on_Mobius11}}, then $H[\alpha]$ is a Seifert-fibered space over the M\"{o}bius band with one exceptional fiber of index $s$.
\end{thm}

\begin{figure}[tbp]
\includegraphics{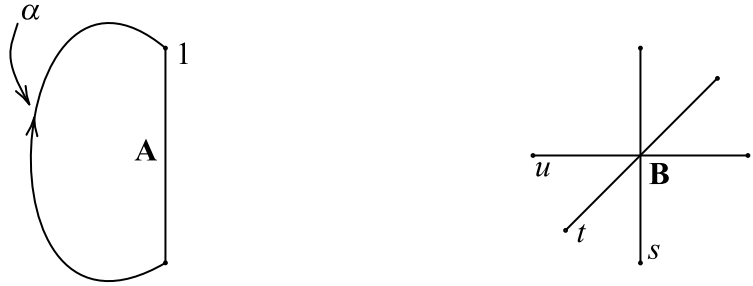}\caption{If $\alpha$ is a primitive curve in the boundary of a genus two handlebody $H$, then $\alpha$ has an R-R diagram with the form of this figure.}\label{PSFFig1aa1}
\end{figure}

\begin{figure}[tbp]
\centering
\includegraphics[width = 1.0\textwidth]{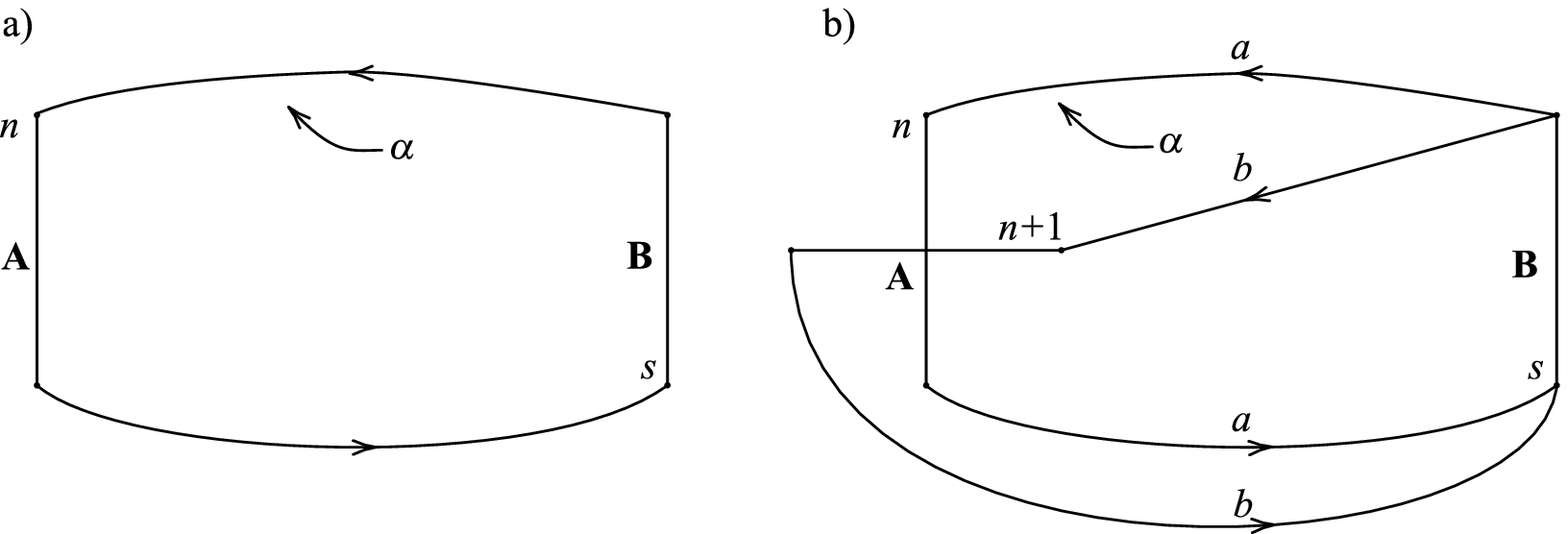}
\caption{If $\alpha$ is a Seifert-d curve in the boundary of a genus two handlebody $H$, then $\alpha$ has an R-R diagram with the form of one of these figures with $n, s > 1$, $a, b > 0$, and $\gcd(a,b) = 1$. If $\alpha$ has an R-R diagram of the form shown in Figure~\ref{PSFFig2a1}a (\ref{PSFFig2a1}b, respectively), then $H[\alpha]$ is a Seifert-fibered space over $D^2$ with two exceptional fibers of indexes $n$ and $s$ ($n(a+b)+b$ and $s$, respectively).}
\label{PSFFig2a1}
\end{figure}

\begin{figure}[tbp]
\centering
\includegraphics[width = 0.55\textwidth]{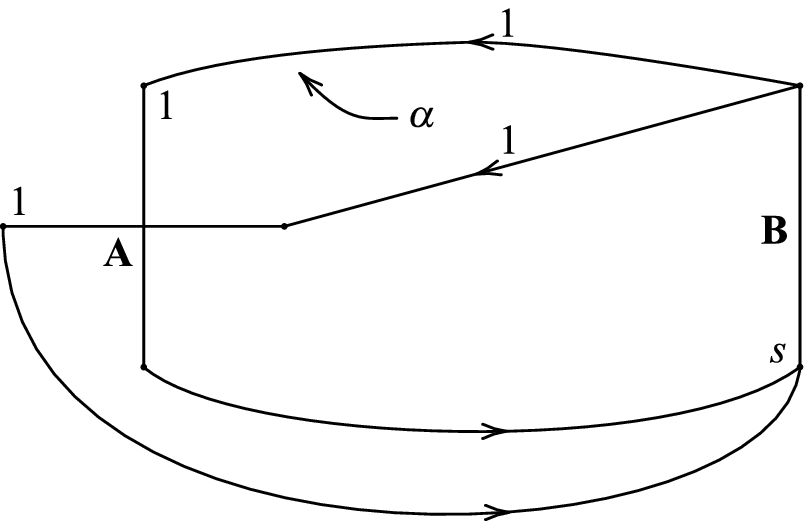}
\caption{If $\alpha$ is a Seifert-m curve in the boundary of a genus two handlebody $H$, then $\alpha$ has an R-R diagram with the form of this figure with $s>1$ in which $\alpha = AB^sA^{-1}B^s$ in $\pi_1(H)$ and $H[\alpha]$ is a Seifert-fibered space over the M\"{o}bius band with one exceptional fiber of index $s$.}
\label{SF_on_Mobius11}
\end{figure}

We can further classify proper power curves in the following theorem. From now on,
to distinguish proper power curves from primitive or Seifert curves, we use the letter $\beta$ instead of $\alpha$ to represent proper power curves. We will see such a situation in Sections~\ref{Genus two R-R diagrams of Seifert-d curves} and \ref{Genus two R-R diagrams of Seifert-m curves}.

\begin{thm}[\textbf{Further classification of proper power curves}]\label{main theorem2}
Suppose $H$ is a genus two handlebody with a complete set of cutting disks $\{D_A, D_B\}$ with $\pi_1(H)=F(A,B)$, where
the generators $A$ and $B$ are dual to $D_A$ and $D_B$ respectively.
If $\beta$ is a proper power curve in $H$, then $\beta$ has one of the following R-R diagrams with respect to the complete set of cutting disks $\{D_A, D_B\}$
up to the homeomorphisms of $H$ inducing the automorphisms
exchanging $A$ and $B$, and replacing $A^{-1}$ by $A$ or $B^{-1}$ by $B$:
\begin{enumerate}
\item \emph{Type I:} $\beta$ has an R-R diagram with a $0$-connection in at least one of the handles.
\item \emph{Type II:} $\beta$ has an R-R diagram of the form shown in Figure~\emph{\ref{PPower8}} with $s>1$.
\item \emph{Type III:} $\beta$ has an R-R diagram of the form shown in Figure~\emph{\ref{PPower10}} with $a,b>0$ and $s>0$.
\item \emph{Type IV:} $\beta$ has an R-R diagram of the form shown in Figure~\emph{\ref{PPower10-1}} with $a,b,c>0$.
\item \emph{Type V:} $\beta$ has an R-R diagram of the form shown in Figure~\emph{\ref{PPower11-1}} with $a,b,c,d>0$.
\end{enumerate}
\end{thm}

\begin{figure}[t]
\centering
\includegraphics[width = 0.7\textwidth]{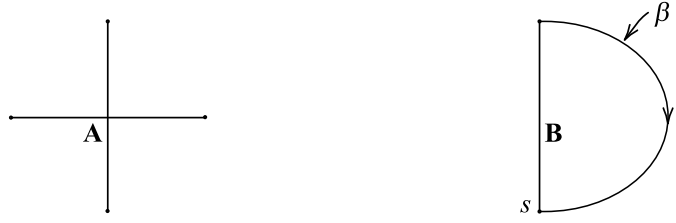}
\caption{Proper power curves referred as $\alpha$ in Theorem~\ref{main theorem} and Type II of proper power curves $\beta$: $[\beta]=B^s$ in Theorem~\ref{main theorem2}.}
\label{PPower8}
\end{figure}

\begin{figure}[t]
\centering
\includegraphics[width = 0.6\textwidth]{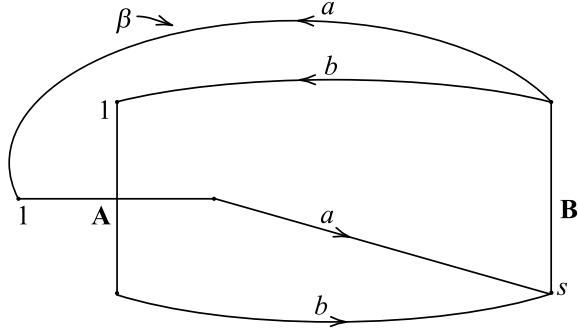}
\caption{Type III of proper power curves $\beta$: $[\beta]=(AB^s)^{a+b}$, where $s>0$ and $a,b>0$.}
\label{PPower10}
\end{figure}

\begin{figure}[t]
\centering
\includegraphics[width = 0.7\textwidth]{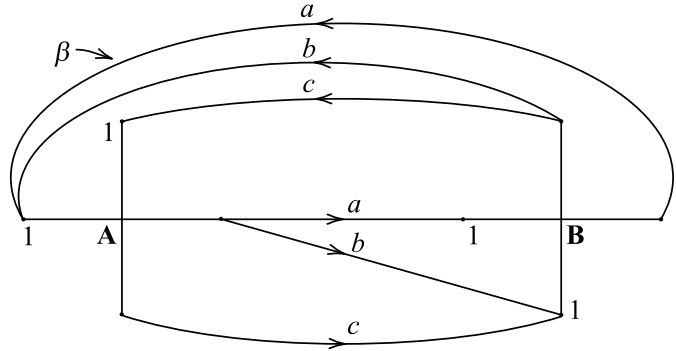}
\caption{Type IV of proper power curves $\beta$: $[\beta]=(AB)^{a+b+c}$, where $a,b,c>0$.}
\label{PPower10-1}
\end{figure}

\begin{figure}[t]
\centering
\includegraphics[width = 0.7\textwidth]{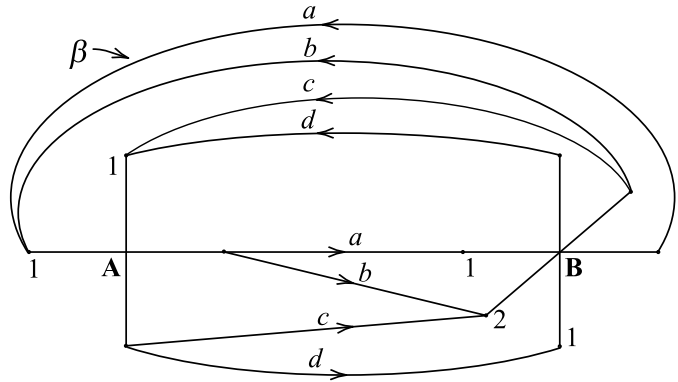}
\caption{Type V of proper power curves $\beta$, where $a,b,c,d>0$.}
\label{PPower11-1}
\end{figure}

The classifications of such curves play very important role in the classification of all hyperbolic primitive/Seifert(or simply P/SF) knots in $S^3$ whose complete list is given in \cite{BK20}. The classifications of primitive and Seifert curves are the first step in the classification of hyperbolic P/SF knots in $S^3$. Additionally if $\alpha$ is Seifert in $H$ such that $H[\alpha]$ embeds in $S^3$, then $H[\alpha]$ is homeomorphic to the exterior of a torus knot. Therefore the classification of Seifert curves in $H$ naturally carries that of all simple closed curves $\alpha$ such that $H[\alpha]$ is homeomorphic to the exterior of a torus knot.

The classification of proper power curves has various applications. First it is used to determine if P/P(primitive/primitive) or P/SF knots are hyperbolic or not. It turns out that if P/P or P/SF knots are not hyperbolic, then there exists a proper power curve in some circumstances. Another application is that when a curve $\alpha$ is Seifert in $H$,
i.e., $H[\alpha]$ is an orientable Seifert-fibered space, a proper power curve disjoint from $\alpha$ becomes a regular fiber of the Seifert-fibered space $H[\alpha]$ and can be used to compute the indexes of exceptional fibers. Also in order to classify some type of primitive/Seifert knots, called knots in Once-Punctured Tori(or simply OPT), properties of proper power curves are used.\\

\noindent\textbf{Acknowledgement.} In 2008, in a week-long series of talks to a seminar in the department of mathematics of the University of Texas as Austin, John Berge outlined a project to completely classify and describe the primitive/Seifert knots in $S^3$. The present paper, which provides some of the background materials necessary to carry out the project, is originated from the joint work with John Berge for the project. I should like to express my gratitude to John Berge for his support and collaboration. I would also like to thank Cameron Gordon and John Luecke for their support while I stayed in the University of Texas at Austin.

\section{R-R diagrams of primitive curves}
\label{S: diagrams of orientable Seifert-fibered spaces}

In this section, we classify R-R diagrams of primitive curves in the boundary of a genus two handlebody.
Let $\alpha$ be a primitive curve in a genus two handlebody $H$. In other words, $H[\alpha]$ is a solid torus.
In order to obtain R-R diagrams of $\alpha$ we use the following lemma which shows equivalent conditions of primitivity. The proof may be found in \cite{W36}, \cite{Z70},
or \cite{G87}.

\begin{lem}\label{primitivity}
The following are equivalent:
\begin{itemize}
\item[(1)] $\alpha$ is primitive in $H$, i.e., $H[\alpha]$ is a solid torus;
\item[(2)] $\alpha$ belongs to a basis for the free group $\pi_1(H)$;
\item[(3)] $\alpha$ is transverse to a properly embedded disk in $H$.
\end{itemize}
\end{lem}

\begin{thm}\label{R-R diagram of primitive curves}
If $\alpha$ is a nonseparating simple closed curve in the boundary of a genus two handlebody $H$ such that $\alpha$ is primitive on $H$,
then $\alpha$ has an R-R diagram with the form of Figure~\emph{\ref{PSFFig1aa1}}.
\end{thm}

\begin{proof}
By Lemma~\ref{primitivity}, there exists a cutting disk $D_A$ in $H$ such that $\alpha$ meets $\partial D_A$ transversely
in a single point. Consider the regular neighborhood $N$ of $\partial D_A \cup \alpha$ in $\partial H$. Then $N$ is a
once-punctured torus which contains $\partial D_A$, and whose boundary $\partial N$ bounds a separating disk of $H$.
Let $N'=\overline{\partial H-N}$. Then it follows by cutting $H$ open along $D_A$
that there exists a unique cutting disk $D_B$ of $H$ up to isotopy whose boundary $\partial D_B$ lies on $N'$.

This partition $\{N, N'\}$ of $\partial H$ with $\partial D_A \subset N$ and $\partial D_B \subset N'$ gives rise to an R-R diagram of $\alpha$ in
which $N$ and $N'$ correspond to the $A$-handle and $B$-handle respectively.
Since $\alpha$ lies in $N$ and intersects $\partial D_A$ in a single point, $\alpha$ has an R-R diagram of the form shown in Figure~{\ref{PSFFig1aa1}}.
This completes the proof.
\end{proof}

\section{R-R diagrams of a proper power curve and more classifications} \label{RRdiagram}

In this section, we classify proper power curves. Let $\beta$ be a proper power curve in the boundary of a genus two handlebody $H$. The following lemma is an easy consequence of Lemma~\ref{primitivity}.

\begin{lem}\label{proper power}
The following are equivalent:
\begin{itemize}
\item[(1)] $\beta$ is a proper power curve in $H$, i.e., $\beta$ is
disjoint from an essential separating disk, does not bound a disk, and is not
primitive in $H$;
\item[(2)] $\beta$ is conjugate to $w^n, \hs n>1,$ of $\pi_1(H)$, where $w$
is a free generator of $\pi_1(H)$;
\item[(3)] There exists a complete set of cutting disks $\{D_A, D_B\}$ of $H$ such that $\beta$ is disjoint to, say, $D_A$ and is transverse to $D_B$ $n>1$ times.
\end{itemize}
\end{lem}

The main results of this section are the following.

\begin{thm}\label{main theorem3}
Suppose $\beta$ is a simple closed curve in the boundary of a genus two handlebody $H$.
If $\beta$ is a proper power curve in $H$, then $\beta$ has an R-R diagram of the form shown in Figure~\emph{\ref{PPower8}}.
\end{thm}

\begin{thm}\label{main theorem4}
Suppose $H$ is a genus two handlebody with a complete set of cutting disks $\{D_A, D_B\}$ with $\pi_1(H)=F(A,B)$, where
the generators $A$ and $B$ are dual to $D_A$ and $D_B$ respectively.
If $\beta$ is a proper power curve in $H$, then $\beta$ has one of the following R-R diagrams with respect to the complete set of cutting disks $\{D_A, D_B\}$
up to the homeomorphisms of $H$ inducing the automorphisms
exchanging $A$ and $B$, and replacing $A^{-1}$ by $A$ or $B^{-1}$ by $B$:
\begin{enumerate}
\item \emph{Type I:} $\beta$ has an R-R diagram with a $0$-connection in at least one of the handles.
\item \emph{Type II:} $\beta$ has an R-R diagram of the form shown in Figure~\emph{\ref{PPower8}} with $s>1$..
\item \emph{Type III:} $\beta$ has an R-R diagram of the form shown in Figure~\emph{\ref{PPower10}} with $a,b>0$ and $s>0$.
\item \emph{Type IV:} $\beta$ has an R-R diagram of the form shown in Figure~\emph{\ref{PPower10-1}} with $a,b,c>0$.
\item \emph{Type V:} $\beta$ has an R-R diagram of the form shown in Figure~\emph{\ref{PPower11-1}} with $a,b,c,d>0$.
\end{enumerate}
\end{thm}

The proof of Theorem~\ref{main theorem3} follows immediately from the definition of a proper power curve. In order to prove Theorem~\ref{main theorem4}, we need the following minor generalization of a result of Cohen, Metzler, and Zimmerman \cite{CMZ81} which allows one to determine easily if a given cyclically reduced word in a free group of rank two is a primitive or a proper power of a primitive. %(here, we revised)

\begin{thm}\cite{CMZ81}
\label{recognizing primitives and proper powers}
Suppose a cyclic conjugate of
\[W = A^{n_1}B^{m_1} \dots A^{n_l}B^{m_l}\]
 is a member of a basis of $F(A,B)$ or a proper power of a member of a basis of $F(A,B)$, where $l \geq 1$ and each indicated exponent is nonzero. Then, after perhaps replacing $A$ by $A^{-1}$ or $B$ by $B^{-1}$, there exists $e > 0$ such that:
\[
n_1 = \dots = n_l = 1,
\quad
\text{and}
\quad
\{m_1, \dots ,m_l\} \subseteq \{e, e+1\},
\]
or
\[
\{n_1, \dots ,n_l\} \subseteq \{e, e+1\},
\quad
\text{and}
\quad
m_1 = \dots = m_l = 1.
\]
\end{thm}

\begin{proof}[The proof of Theorem~\emph{\ref{main theorem4}}]

In order to find all possible R-R diagrams of a proper power curve $\beta$, we consider the following cases separately.
\begin{enumerate}
\item $\beta$ has a 0-connection in at least one of the handles in its R-R diagram.
\item $\beta$ has no 0-connections in either handle in its R-R diagram.
\end{enumerate}

The case where $\beta$ has a $0$-connection in one handle gives restriction to other curves in $\partial H$ disjoint from $\beta$. In other words, if
$\gamma$ is a simple closed curve in $\partial H$ disjoint from $\beta$, then $\gamma$ must have only bands of connections labeled by $0$ or $1$ in
that handle. Therefore we put this case into one type of possible proper power curves, which gives Type I in Theorem~\ref{main theorem}.

Now we assume that $\beta$ has no 0-connections in either handle.

Suppose $\beta$ has only one generator in $\pi_1(H)=F(A, B)$. Then up to replacement of $A$ with $A^{-1}$, $B$ with $B^{-1}$, or exchange of $A$ and $B$, we may
assume that $[\beta]=B^s$ for some $s>1$. Since $\beta$ has no 0-connections, this implies that
$\beta$ has no connections in the $A$-handle and only one connection in the $B$-handle.
Thus this case gives rise to Type II of a proper power curve in Theorem~\ref{main theorem3} with an R-R diagram of the form shown in
Figure~\ref{PPower8}.

Now suppose that $\beta$ has the two generators $A$ and $B$ in $\pi_1(H)$. By Theorem~\ref{recognizing primitives and proper powers},
we may assume that $\beta = AB^{m_1} \cdots AB^{m_l}$, where $\{m_1, \dots ,m_l\} \subseteq \{s, s+\epsilon\}$
with $\epsilon=\pm1$ and min$\{s, s+\epsilon\}>0$. This implies that every connection in the $A$-handle is labeled by
1. There are two cases to consider:
\begin{itemize}
\item [(1)] $\beta$ has only one band of 1-connections in the $A$-handle,
\item [(2)] $\beta$ has two bands of 1-connections in the $A$-handle.
\end{itemize}

\begin{figure}[t]
\centering
\includegraphics[width = 0.7\textwidth]{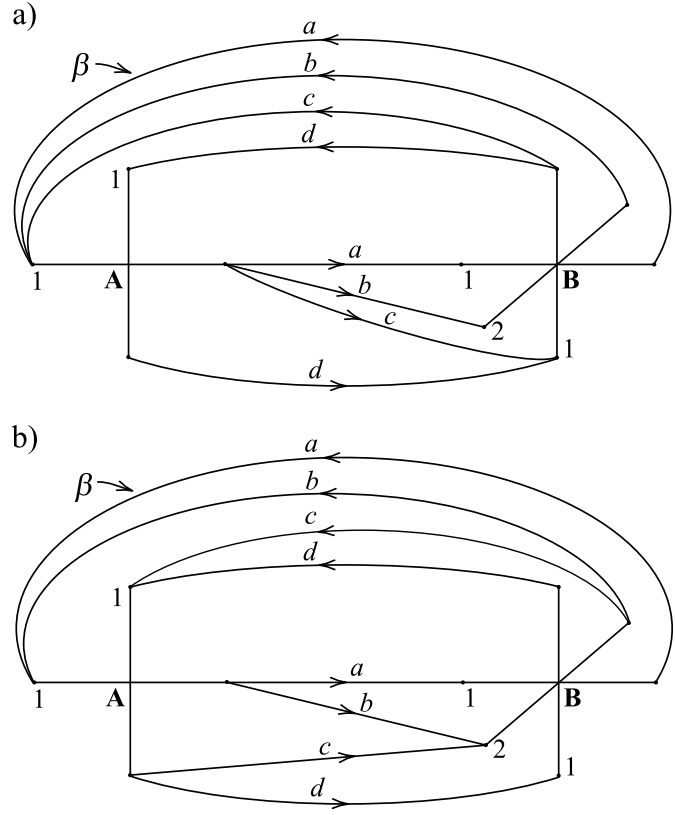}
\caption{An R-R diagram of $\beta$ when there are three bands of connections in the $B$-handle.}
\label{PPower11}
\end{figure}

\textbf{Case (1)}: Assume that $\beta$ has only one band of 1-connections in the $A$-handle.

Suppose $\{m_1, \dots ,m_l\} \subsetneq \{s, s+\epsilon\}$. Without loss of
generality, we may assume that $\{m_1, \dots ,m_l\} = \{s\}$. If $s>1$, then there must be only one band of $s$-connections
in the $B$-handle, in which case $\beta$ is a primitive curve with $[\beta]=AB^s$. Therefore $s=1$ and there
must be two bands of 1-connections in the $B$-handle. Then with $A$ and $B$ exchanged, this case belongs to Type III
of a proper power curve in Theorem~\ref{main theorem3} with an R-R diagram of the form shown in Figure~\ref{PPower10}.

Suppose $\{m_1, \dots ,m_l\} = \{s, s+\epsilon\}$. Then $\beta$ has at least two bands of connections
in the $B$-handle. If $\beta$ has two bands of connections in the $B$-handle, $\beta$ must have an R-R diagram
of the form shown in Figure~\ref{PPower5} and by Lemma~\ref{Notppower1}, $\beta$ is not a proper power
curve. If $\beta$ has three bands of connections in the $B$-handle, then these three bands of connections should be labeled by
1, 2 and 1 respectively, and $\beta$ has the R-R diagram of the form shown in Figure~\ref{PPower7}. By Lemma~\ref{Notppower2},
$\beta$ is not a proper power curve.

\textbf{Case (2)}: $\beta$ has two bands of 1-connections in the $A$-handle.

Suppose $\{m_1, \dots ,m_l\} \subsetneq \{s, s+\epsilon\}$. Without loss of
generality, we may assume that $\{m_1, \dots ,m_p\} = \{s\}$. Then if $\beta$ has only one band of connections in the $B$-handle,
then the band must be labeled by $s$, in which case Type III of a proper power curve
in Theorem~\ref{main theorem3} arises with an R-R diagram of the form shown in Figure~\ref{PPower10}.
If $\beta$ has two bands of connections in the $B$-handle, then $s$ must be 1 and this case yields Type IV of
a proper power curve in Theorem~\ref{main theorem3} with an R-R diagram of the form shown in Figure~\ref{PPower10-1}.

Suppose $\{m_1, \dots ,m_l\} = \{s, s+\epsilon\}$. If there are only two bands of connections
in the $B$-handle, then $\beta$ must have R-R diagram of the form shown in Figure~\ref{PPower3}
and by Lemma~\ref{Notppower3}, $\beta$ is not a proper power curve.

If there are three bands of connections in the $B$-handle, then $\beta$ has two types of R-R diagrams as shown in Figure~\ref{PPower11}.
By Lemma~\ref{Notppower4}, the R-R diagram of $\beta$ in Figure~\ref{PPower11}a cannot be a proper power curve. Therefore
we have the R-R diagram of $\beta$ in Figure~\ref{PPower11}b, which gives Type V of a proper power curve
in Theorem~\ref{main theorem3}.

 Thus we complete the proof of Theorem~\ref{main theorem3}
\end{proof}

\begin{lem}\label{Notppower1}
Suppose a simple closed curve $\beta$ has an R-R diagram with the form shown in Figure~\emph{\ref{PPower5}} where $s>0, \epsilon=\pm1$ with min$\{s, s+\epsilon\}>0$, and $a,b>0$.
Then $\beta$ is a primitive curve.
\end{lem}

\begin{figure}[t]
\centering
\includegraphics[width = 0.6\textwidth]{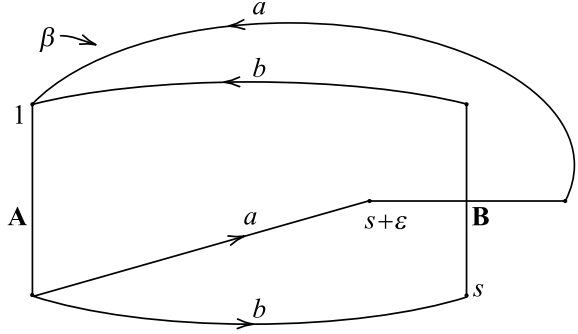}
\caption{An R-R diagram of $\beta$ where $s>0, \epsilon=\pm1$, and $a,b>0$.}
\label{PPower5}
\end{figure}

\begin{proof}
Note that since $\beta$ is a simple closed curve, gcd$(a,b)=1$. Let $b=\rho a+\eta$, where
$\rho\geq 0$ and $0\leq \eta<a$ (if $\eta=0$, then $\rho>0$ and $a=1$).
Now we record the curve $\beta$ algebraically by starting the $a$ parallel arcs, i.e., the band of width $a$, entering into the
$(s+\epsilon)$-connection in the $B$-handle. It follows that $\beta$ is the product
of two subwords $AB^{s+\epsilon}(AB^s)^{\rho}$ and $AB^{s+\epsilon}(AB^s)^{\rho+1}$ with $|AB^{s+\epsilon}(AB^s)^{\rho}|=a-\eta$
and $|AB^{s+\epsilon}(AB^s)^{\rho+1}|=\eta$. Here, for example, $|AB^{s+\epsilon}(AB^s)^{\rho}|$ denotes the total number of appearances of $AB^{s+\epsilon}(AB^s)^{\rho}$ in the word of $\beta$ in $\pi_1(H)=F(A,B)$.

There is a change of cutting disks of the handlebody
$H$, which induces an automorphism of $\pi_1(H)$
that takes $A \mapsto AB^{-s}$ and leaves $B$ fixed. Then by this change of cutting disks,
$AB^{s+\epsilon}(AB^s)^{\rho}$ and $AB^{s+\epsilon}(AB^s)^{\rho+1}$ are sent to $A^{\rho+1}B^{\epsilon}$ and $A^{\rho+2}B^{\epsilon}$
respectively. Therefore the resulting Heegaard diagram of $\beta$ realizes a new R-R diagram of the form
in Figure~\ref{PPower6}, where the positions of the $A$ and $B$-handles are switched.
The new R-R diagram has the same form as that in Figure~\ref{PPower5} with less number of
arcs. One can continue inductively until one of the labels of parallel arcs is 0,
in which case $[\beta]=AB^j$ for some $j> 0$ up to replacement of $A$ with $A^{-1}$, $B$ with $B^{-1}$, or exchange of $A$ and $B$.
Such a curve $\beta$ is a primitive curve.
\end{proof}

\begin{figure}[t]
\centering
\includegraphics[width = 0.6\textwidth]{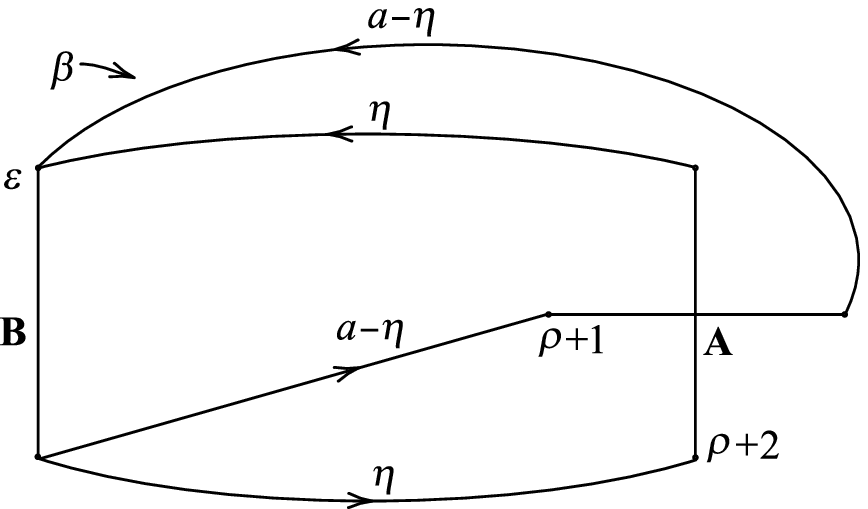}
\caption{The R-R diagram after change of cutting disks of the handlebody
$H$ inducing an automorphism of $\pi_1(H)$
that takes $A \mapsto AB^{-s}$ and leaves $B$ fixed.}
\label{PPower6}
\end{figure}
%%%%%%%%%%%%%%%%%%%%%%%%%%%%%%%%%%%%%%%%%%%%%%%%%%%%%%%%%%%%%%%%%%%%%%%%%%%%%%%%%%%%%%%%%%%%%%%%%%%%%%%%%%%%%%%%%%%%%%%%%%%%%%%%%%%%
\begin{lem}\label{Notppower2}
Suppose a simple closed curve $\beta$ has an R-R diagram with the form shown in Figure~\emph{\ref{PPower7}} where $a,b,c >0$.
Then $\beta$ is not a proper power curve.
\end{lem}

\begin{figure}[t]
\centering
\includegraphics[width = 0.6\textwidth]{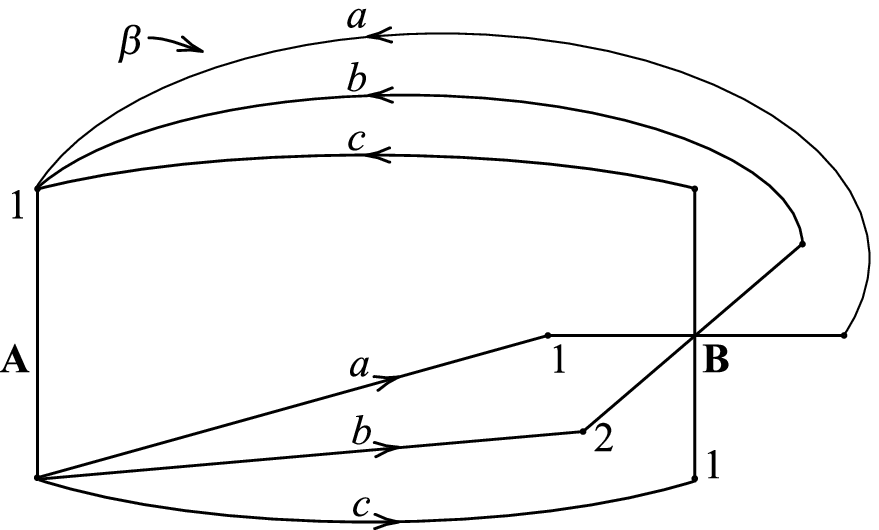}
\caption{An R-R diagram of $\beta$ where $a,b,c>0$.}
\label{PPower7}
\end{figure}

\begin{proof}
First, note that gcd$(a+b, b+c)=1$ in order for $\beta$ to be a simple closed curve.
Since $AB$ and $AB^2$ appear in $[\beta]$, by considering $\{AB, AB^2\}$ as a generating set of $F(A, B)$ and
by Theorem~\ref{recognizing primitives and proper powers} one of the two must appear with exponent $1$.
However, from the R-R diagram one can see that if one reads $\beta$ from the right-hand edge
of the band of width $a$ entering the 1-connection in the $B$-handle, then $AB$ appears twice consecutively. Therefore $AB^2$ must
have only exponent $1$.

If $b=1$, then $AB^2$ appears only once and thus $\beta$ is a primitive curve.
Suppose $b>1$. Consider following the band of width $b$ around the R-R diagram.
Since $\beta$ is a simple closed curve, this band must split into two subbands at the endpoint
of the $(-1)$-connection in the $A$-handle. If the band splits into two subbands belonging
to the bands of width $a$ and $b$ respectively, then by tracing arcs
of this band we see that $\beta$ has two subwords $AB^2(AB)^n$ and $AB^2(AB)^{n+e}$ for some
positive integers $n$ and $e$ with $e>1$. It follows from Theorem~\ref{recognizing primitives and proper powers} that $\beta$ cannot be a proper power curve.
Similarly for the case where the band splits into two subbands belonging
to the bands of width $b$ and $c$ respectively.
\end{proof}

%%%%%%%%%%%%%%%%%%%%%%%%%%%%%%%%%%%%%%%%%%%%%%%%%%%%%%%%%%%%%%%%%%%%%%%%%%%%%%%%%%%%%%%%%%%%%%%%%%%%%%%%%%%%%%%%%%%%%%%%%%%%%%%%%%%%

\begin{lem}\label{Notppower3}
Suppose a simple closed curve $\beta$ has an R-R diagram with the form shown in Figure~\emph{\ref{PPower3}} where $s>0, \epsilon=\pm1$ with min$\{s, s+\epsilon\}>0$, and $a,b,c>0$.
Then $\beta$ is not a proper power curve.
\end{lem}

\begin{figure}[t]
\centering
\includegraphics[width = 0.7\textwidth]{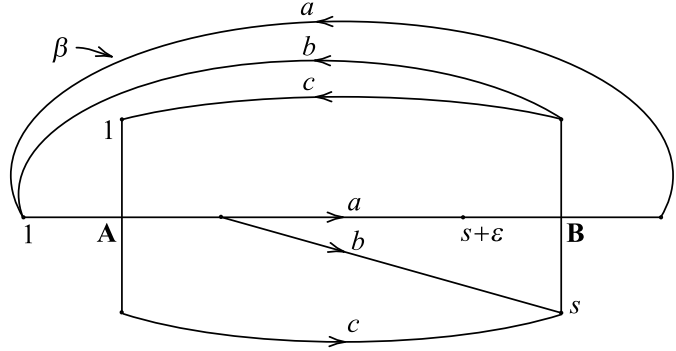}
\caption{An R-R diagram of $\beta$ where $s>0, \epsilon=\pm1$, and $a,b,c>0$.}
\label{PPower3}
\end{figure}

\begin{figure}[t]
\centering
\includegraphics[width = 1\textwidth]{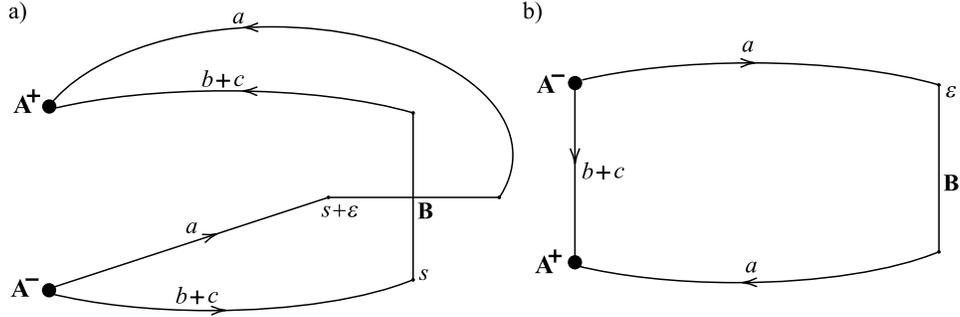}
\caption{The hybrid diagram of $\beta$ and change of cutting disks of $H$.}
\label{PPower4}
\end{figure}

\begin{proof}
We use the argument of hybrid diagrams which are introduced in \cite{K20}.
Consider the corresponding hybrid diagram as shown in Figure~\ref{PPower4}a.
Then we drag the vertex \textbf{A}$^{\boldsymbol{-}}$ together with
the edges meeting the vertex \textbf{A}$^{\boldsymbol{-}}$ over the $s$-connection in the $B$-handle. This corresponds to the
change of the cutting disks of $H$ inducing an automorphism of $\pi_1(H)$ which takes $A\mapsto AB^{-s}$ and leaves $B$ fixed. The resulting hybrid diagram is shown
in Figure~\ref{PPower4}b.

Transforming the hybrid diagram in Figure~\ref{PPower4}b back into an R-R diagram, there are three possible cases to consider as follows:

\begin{enumerate}
\item There is only one band of connections in the $A$-handle;
\item There are only two bands of connections in the $A$-handle;
\item There are three bands of connections in the $A$-handle.
\end{enumerate}

Note from Figure~\ref{PPower4}b that since $b+c>0$, all of the labels of bands of connections in the $A$-handle are greater than $0$ and
at least one of the bands of connections has label greater than $1$.

(1) Suppose that there is only one band of connections in the $A$-handle. Then $a$ must be 1 and $\beta=B^{\epsilon}A^{b+c+1}$, which is primitive.

(2) Suppose that there are only two bands of connections in the $A$-handle. Let $p$ and $q$ be the labels of the two bands of
connections. Since $b+c>0$ and thus at least one of the bands of connections has label greater than $1$, $p\neq q$.
Theorem~\ref{recognizing primitives and proper powers} forces $|p-q|$ to be $1$. However, by Lemma~\ref{Notppower1} such a curve $\beta$
is a primitive curve.

(3) Suppose that there are three bands of connections in the $A$-handle. Let $p, q,$ and $r$ be the labels of the three bands of
connections with $q=p+r$. In order for $\beta$ to be a proper power, $(p,q,r)=(1,2,1)$. However, by Lemma~\ref{Notppower2}, $\beta$ is not a proper power curve.
\end{proof}

%%%%%%%%%%%%%%%%%%%%%%%%%%%%%%%%%%%%%%%%%%%%%%%%%%%%%%%%%%%%%%%%%%%%%%%%%%%%%%%%%%%%%%%%%%%%%%%%%%%%%%%%%%%%%%%%%%%%%%%%%%%%%%%%%%%%

\begin{lem}\label{Notppower4}
Suppose a simple closed curve $\beta$ has an R-R diagram with the form shown in Figure~\emph{\ref{PPower1}} where $a,b,c,d>0$.
Then $\beta$ is not a proper power curve.
\end{lem}

\begin{figure}[t]
\centering
\includegraphics[width = 0.7\textwidth]{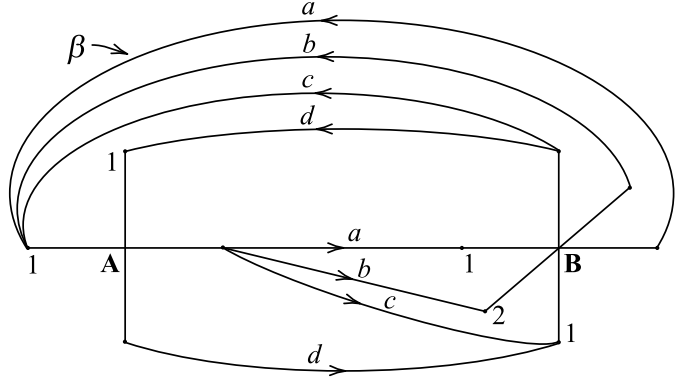}
\caption{An R-R diagram of $\beta$ with $a,b,c,d>0$.}
\label{PPower1}
\end{figure}

\begin{proof}
Consider the corresponding hybrid diagram as shown in Figure~\ref{PPower2}a.

Then we drag the vertex \textbf{A}$^{\boldsymbol{-}}$ together with
the edges meeting the vertex \textbf{A}$^{\boldsymbol{-}}$ over the $2$-connection in the $B$-handle. This corresponds to the
change of the cutting disks of $H$ inducing an automorphism of $\pi_1(H)$ which takes $A\mapsto AB^{-2}$ and leaves $B$ fixed. The resulting hybrid diagram is shown
in Figure~\ref{PPower2}b.

Transforming the hybrid diagram in Figure~\ref{PPower2}b back into an R-R diagram, there are three possible cases to consider as follows:

\begin{enumerate}
\item There is only one band of connections in the $A$-handle;
\item There are only two bands of connections in the $A$-handle;
\item There are three bands of connections in the $A$-handle.
\end{enumerate}

Note from Figure~\ref{PPower2}b that since $b>0$, all of the labels of bands of connections in the $A$-handle are greater than $0$
and at least one of the bands of connections has label greater than $1$.

(1) Suppose that there is only one band of connections in the $A$-handle. Let $p$ be the label of the band of connections.
Since there are the $b$ edges connecting \textbf{A}$^{\boldsymbol{+}}$ and \textbf{A}$^{\boldsymbol{-}}$, $p>1$. On the other hand, in Figure~\ref{PPower1}, consider chasing
back and forth the outermost arc in the band of width $a$ entering the 1-connection in the $A$-handle. This represents
$ABAB\cdots$, which is carried into $AB^{-1}AB^{-1}\cdots$ by the automorphism $A\mapsto AB^{-2}$. This implies
that $\beta$ has a 1-connection in the $A$-handle, a contradiction.

\begin{figure}[t]
\centering
\includegraphics[width = 1\textwidth]{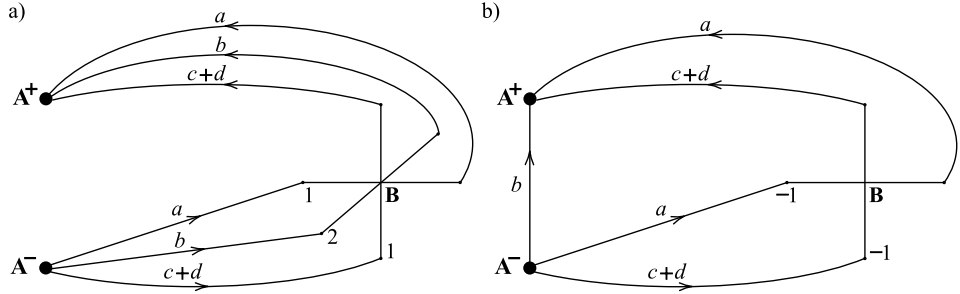}
\caption{The hybrid diagram of $\beta$ and the change of cutting disks of $H$.}
\label{PPower2}
\end{figure}

(2) Suppose that there are only two bands of connections in the $A$-handle. Let $p$ and $q$ be the labels of the two bands of
connections. By the similar argument in the proof of Lemma~\ref{Notppower3}, $|p-q|=1$. However, by Lemma~\ref{Notppower3}
such a curve $\beta$ cannot be a proper power.

(3) Suppose that there are three bands of connections in the $A$-handle. Let $p, q,$ and $r$ be the labels of the three bands of
connections with $q=p+r$. In order for $\beta$ to be a proper power, $(p,q,r)=(1,2,1)$. Then by switching the $A$-
and $B$-handles, and the signs of the labels, the R-R diagram of $\beta$
has the same form as in Figure~\ref{PPower1} with less number of the arcs connecting the $A$- and $B$-handles.
So if we can continue to perform the change of cutting disks of $H$, then since the number of the edges is strictly decreasing
under the change of cutting disks, this case must eventually belong to the cases (1) and (2).
\end{proof}

\section{R-R diagrams of Seifert-d curves}
\label{Genus two R-R diagrams of Seifert-d curves}

In this section, we classify the R-R diagrams of a simple closed curve which is Seifert-d in a genus two handlebody.
If $\alpha$ is a Seifert-d curve in a genus two handlebody $H$, then by its definition $H[\alpha]$ is a Seifert-fibered space over $D^2$ with two exceptional fibers. In order to compute the type of exceptional fibers in $H[\alpha]$, we need the following notation and lemma.

\textit{Notation}. If $U = (a,b)$ is an element of $\mathbb{Z} \oplus \mathbb{Z}$, let $U^{\perp}$ denote the element $(-b,a)$ of $\mathbb{Z} \oplus \mathbb{Z}$, and let
 `$\circ$' denote the usual \emph{inner product} or \emph{dot product} of vectors.

\begin{lem}
\label{dot product}
Let $U = (a,b)$, $V = (c,d)$ and $W = (e,f)$ be three elements of $\mathbb{Z} \oplus \mathbb{Z}$ such that
$ad-bc = \pm1$. If $W$ is expressed as a linear combination of $U$ and $V$, say $W = xU + yV$, then $y = \pm (U^{\perp} \circ W)$.
\end{lem}

\begin{proof}
We take an inner product by $U^{\perp}$ on both sides of $W = xU + yV$ . Then since $U^{\perp}\circ U=0$ and $U^{\perp}\circ V=\pm1$, the result follows.
\end{proof}

The main result of this section is the following theorem.

\begin{thm}
\label{Thm describing R-R diagrams of SF spaces over the disk}
If $\alpha$ is a nonseparating simple closed curve in the boundary of a genus two handlebody $H$ such that $H[\alpha]$ is Seifert-fibered over $D^2$ with two exceptional fibers, then $\alpha$ has an R-R diagram with the form of Figure~\emph{\ref{PSFFig2a}}\emph{a} with $n, s > 1$, or \emph{\ref{PSFFig2a}}\emph{b} with $n, s > 1$, $a, b > 0$, and $\gcd(a,b) = 1$.

Conversely, if $\alpha$ has an R-R diagram with the form of Figure~\emph{\ref{PSFFig2a}}\emph{a} with $n, s > 1$, or Figure \emph{\ref{PSFFig2a}}\emph{b} with $n, s > 1$, $a, b > 0$, and $\gcd(a,b) = 1$, then $H[\alpha]$ is Seifert-fibered over $D^2$ with two exceptional fibers of indexes $n$
and $s$ in Figure~\emph{\ref{PSFFig2a}}\emph{a} or indexes $n(a+b)+b$ and $s$ in Figure~\emph{\ref{PSFFig2a}}\emph{b}.

In addition, a curve $\beta$ shown in Figure~\emph{\ref{PSFFig2b}}, which is an augmentation of Figure~\emph{\ref{PSFFig2a}}, is a regular fiber of $H[\alpha]$.
\end{thm}

\begin{figure}[tbp]
\centering
\includegraphics[width = 1.0\textwidth]{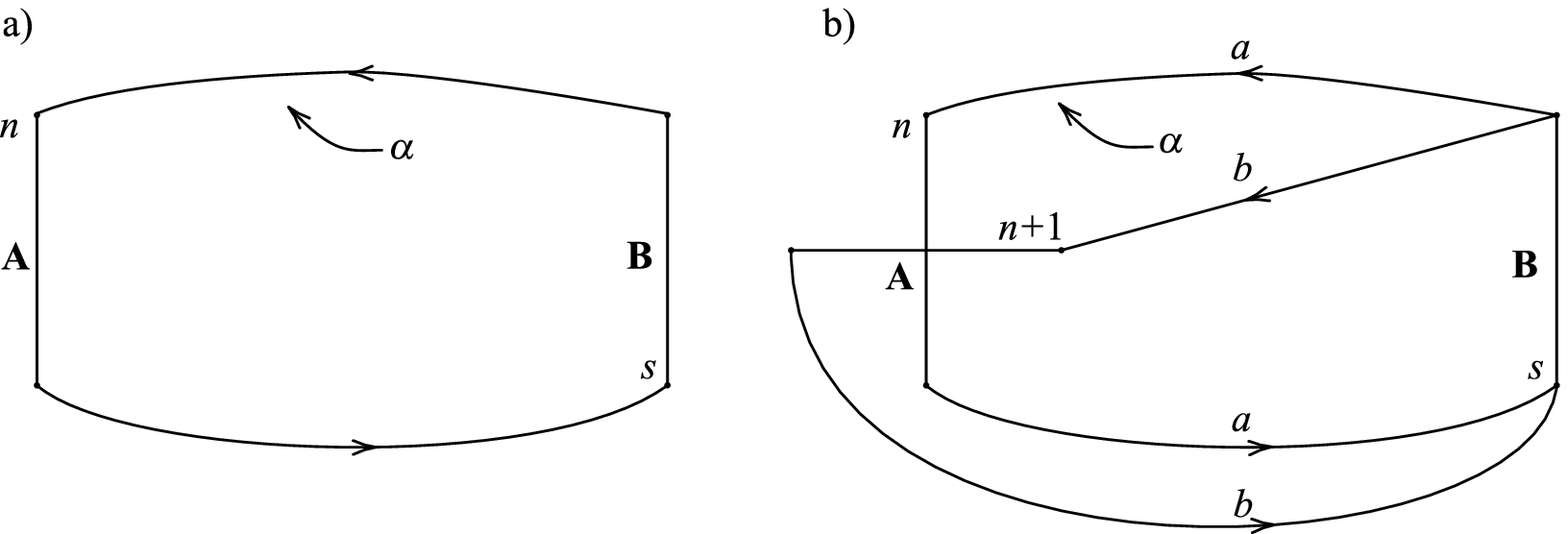}
\caption{If $\alpha$ is a nonseparating simple closed curve in the boundary of a genus two handlebody $H$ such that $H[\alpha]$ is Seifert-fibered over $D^2$ with two exceptional fibers, then $\alpha$ has an R-R diagram with the form of one of these figures with $n, s > 1$, $a, b > 1$, and $\gcd(a,b) = 1$. The converse also holds. (See Figure \ref{PSFFig2b}.)}
\label{PSFFig2a}
\end{figure}

\begin{proof}
We start by showing that if $H[\alpha]$ is Seifert-fibered over $D^2$ with two exceptional fibers, then $\alpha$ has an R-R diagram of the claimed form.

The key idea is that $H[\alpha]$, which is defined to be the manifold obtained by
adding a 2-handle to $H$ along $\alpha$, induces a genus two Heegaard decomposition of a Seifert-fibered space over $D^2$ with two exceptional fibers.
However Heegaard decompositions of a Seifert-fibered space over $D^2$ are well understood. For instance, Theorem~\ref{SF Heegaard diagram description from BRZ88}, of Boileau, Rost and Zieschang, completely describes the genus two Heegaard diagrams of a Seifert-fibered space over $D^2$ with two exceptional fibers. Using this result, Theorem~\ref{T detailing the correspondence between Heegaard and R-R diagrams of SF spaces over the disk} shows how the Heegaard diagrams described in Theorem~\ref{SF Heegaard diagram description from BRZ88} translate into R-R diagrams. And then Lemma~\ref{May assume n > 1 in R-R diagrams of SF over the disk} adds a finishing detail to the proof of this direction by showing that it is always possible to assume that $n > 1$ in Figure~\ref{PSFFig2a}b.

With the proof of one direction of Theorem~\ref{Thm describing R-R diagrams of SF spaces over the disk} finished, it remains to show that if $\alpha$ has an R-R diagram with the form of Figure~\ref{PSFFig2a}a with $n, s > 1$, or Figure \ref{PSFFig2a}b with $n, s > 1$, $a, b > 0$, and $\gcd(a,b) = 1$, then $H[\alpha]$ is Seifert-fibered over $D^2$ with two exceptional fibers.

To see that this is the case, consider Figure~\ref{PSFFig2b} in which each of the R-R diagrams of Figure \ref{PSFFig2a} has been augmented with a simple closed curve $\beta$ disjoint from $\alpha$. Then, in each diagram of Figure~\ref{PSFFig2b}, two parallel copies of $\beta$ bound an essential separating annulus $\mathcal{A}$ in $H$. (Figure~\ref{PSFFig2c} illustrates the situation when $s = 2$, and $D_A$ and $D_B$ are cutting disks of $H$ underlying the A-handle and B-handle of the R-R diagram of $\alpha$.)

Cutting $H$ apart along $\mathcal{A}$ yields a genus two handlebody $W$ and a solid torus $V$. Note that $\alpha$ lies in $\partial W$ as a primitive curve in $W$ implying that $W[\alpha]$ is a solid torus, because any component of $D_B\cap W$ is a cutting disk $D_C$ of $W$ such that an R-R diagram of $\alpha$ with respect to $\{D_A, D_C\}$ of $W$ has the form in Figure~\ref{PSFFig2a} with $s$ replaced by $1$, in which case $\alpha$ in Figure~\ref{PSFFig2a}a
intersects $D_C$ only once and thus is primitive, and $\alpha$ in Figure\ref{PSFFig2a}b is primitive by Lemma~\ref{Notppower1}.
It follows that $H[\alpha]$ is obtained by gluing the two solid tori $W[\alpha]$ and $V$ together along $\mathcal{A}$. So $H[\alpha]$ is Seifert-fibered over $D^2$ with $\beta$ as a regular fiber and the cores of $W[\alpha]$ and $V$ as exceptional fibers.

The last step is to compute the indexes of the two exceptional fibers of $H[\alpha]$. It is clear that %in each case,
the annulus $\mathcal{A}$ wraps around the solid torus $V$ $s$ times longitudinally, so the core of $V$ is an exceptional fiber of index $s > 1$.
For the other index, it follows by computing $\pi_1((W[\alpha])[\beta])(=\pi_1((W[\beta])[\alpha])=\mathbb{Z}_n)$
that if $\alpha$ has an R-R diagram with the form of Figure~\ref{PSFFig2a}a, then the core of $W[\alpha]$ is an exceptional fiber of index $n$.

Finally, suppose the R-R diagram of $\alpha$ has the form of Figure~\ref{PSFFig2a}b. In this case, Lemma~\ref{dot product} can be used to compute the index of the second exceptional fiber formed by the core of $W[\alpha]$.

Abelianizing $\pi_1(W)$, we have:
$[\vec{\alpha}]$ = $(n(a+b) + b, a+b)$ and $[\vec{\beta}\,]$ = $(0, 1)$. By Lemma~\ref{dot product}:
\[
[\vec{\beta}\,] = \pm ([\vec{\alpha}]^{\perp} \circ [\vec{\beta}\,]) = \pm ((-(a+b), n(a+b) + b)\circ (0, 1)) = \pm (n(a+b) + b)
\]
in $H_1(W[\alpha])$.
Thus the regular fiber $\beta$ wraps around the solid torus $W[\alpha]$ longitudinally $n(a+b) + b > 1$ times, so the core of $W[\alpha]$ is an exceptional fiber of index $n(a+b) + b$.
\end{proof}

\begin{rem}\label{about regular fibers}
(1) If $\alpha$ has an R-R diagram of the form shown in Figure~\ref{PSFFig2a}a(Figure~\ref{PSFFig2a}b, resp.), $\alpha$ is said to be a Seifert-d curve \textit{of rectangular(non-rectangular, resp.) form}.

(2) The regular fiber $\beta$ is a proper power curve representing $B^s$ in $\pi_1(H)=F(A, B)$, where $A$ and $B$ are dual to the cutting disks $D_A$
and $D_B$ underlying the $A$-handle and $B$-handle respectively of the R-R diagram.
In addition, since the R-R diagram of Figure~\ref{PSFFig2a}a is symmetric, the argument of the fifth paragraph in the proof above shows that
that the proper power curve representing $A^n$ in the R-R diagram of Figure~\ref{PSFFig2a}a is also a regular fiber of $H[\alpha]$.
\end{rem}

\begin{figure}[tbp]
\centering
\includegraphics[width = 1.0\textwidth]{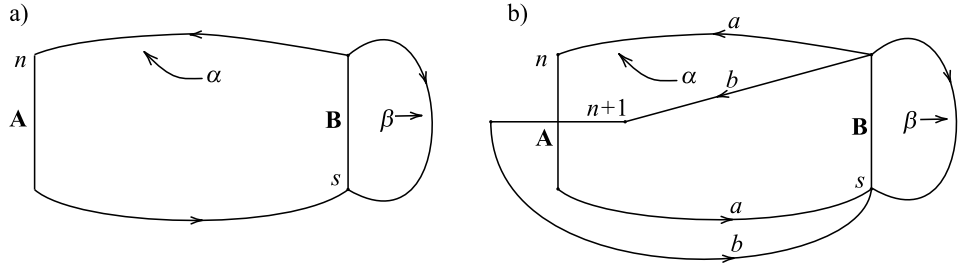}
\caption{In this figure, each of the R-R diagrams of Figure \ref{PSFFig2a} has been augmented with a simple closed curve $\beta$ disjoint from $\alpha$. Then, in each case, two parallel copies of $\beta$ bound an essential separating annulus $\mathcal{A}$ in $H$. Cutting $H$ apart along $\mathcal{A}$ yields a genus two handlebody $W$ and solid torus $V$. Then $\alpha$ lies in $\partial W$ and $W[\alpha]$ is a solid torus. It follows that $H[\alpha]$ is obtained by gluing two solid tori together along $\mathcal{A}$. So $H[\alpha]$ is Seifert-fibered over $D^2$ with $\beta$ as regular fiber and the cores of $W[\alpha]$ and $V$ as exceptional fibers. As for the indexes of the exceptional fibers: If the R-R diagram of $\alpha$ has the form of Figure \ref{PSFFig2a}a, the indexes are $n$ and $s$ respectively. If the R-R diagram of $\alpha$ has the form of Figure \ref{PSFFig2a}b, the indexes are $n(a+b) + b$ and $s$ respectively.}
\label{PSFFig2b}
\end{figure}

\begin{figure}[tbp]
\centering
\includegraphics[width = 0.6\textwidth]{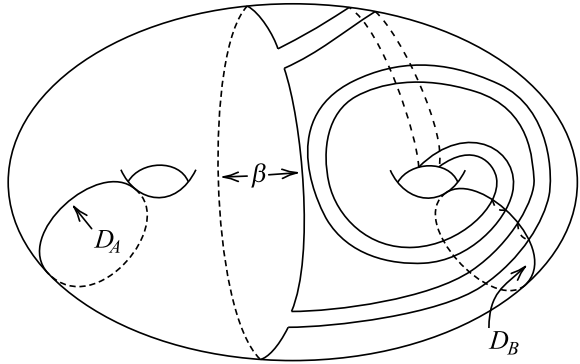}
\caption{If $H$ is a handlebody of genus two with cutting disks $D_A$ and $D_B$ and $\beta$ is a nonseparating simple closed curve in $\partial H$ such that $|\beta \cap \partial D_A| = 0$, and $|\beta \cap \partial D_B| = s > 1$, then two parallel copies of $\beta$ bound an essential separating annulus $\mathcal{A}$ in $H$. This figure illustrates $D_A$, $D_B$, and $\mathcal{A}$ in the special case $s = 2$.}
\label{PSFFig2c}
\end{figure}

Turning to the description of Heegaard decompositions of orientable Seifert-fibered spaces over $D^2$ with two exceptional fibers, let $S(\nu/p, \omega/ q)$ denote an orientable Seifert-fibered space over the disk $D^2$ which has two exceptional fibers of types $\nu / p$ and $\omega / q$ with $0 < \nu < p$ and $0 < \omega < q$.

Also let $W_{m,n}(x,y)$ be the unique primitive word up to conjugacy in the free group $F(x,y)$ which has $(m,n)$ as its abelianization. Then, if $v$ and $w$ are words in $x$ and $y$, $W_{m,n}(v,w)$ is the word obtained from $W_{m,n}(x,y)$ by substituting $v$ for $x$ and $w$ for $y$ in $W_{m,n}(x,y)$.

Then the following theorem of Boileau, Rost and Zieschang(Theorem 5.4 in \cite{BRZ88}) completely describes the genus two Heegaard diagrams of $S(\nu/p, \omega/ q)$.
For the notations in Theorem~\ref{SF Heegaard diagram description from BRZ88}, see Sections 2, 4, and 5 in \cite{BRZ88}.

\begin{thm}\cite{BRZ88}
\label{SF Heegaard diagram description from BRZ88}
The manifold $S(\nu/p, \omega/ q)$ admits three genus two Heegaard decompositions $HD_0$, $HD_S$, and $HD_T$, represented by the following Heegaard diagrams:
\[
HD_0 \leftrightarrow (s^pt^{-q}; \lambda, \mu), \: HD_S \leftrightarrow (W_{p,\nu}(u^{-1}, t^q); -, \mu), \: HD_T \leftrightarrow (W_{q,\omega}(v^{-1}, s^p); \lambda, -).
\]
Here $\nu \lambda \equiv 1 \mod p$ and $\omega \mu \equiv  1 \mod q$. Any Heegaard decomposition of genus two of $S(\nu /p, \omega/ q)$ is homeomorphic to one of these. Moreover:
\begin{enumerate}
\item $HD_0$ is homeomorphic to $HD_T$ \emph{(}or $HD_S$\emph{)} if and only if $\omega \equiv  \pm 1 \mod q$ \emph{(} or $\nu \equiv  \pm 1 \mod p$, respectively\emph{)}.
\item If $\omega\equiv  \pm 1 \mod q$ and $\nu \equiv  \pm 1 \mod p$, then $HD_0$, $HD_S$, and $HD_T$ are all homeomorphic.
\item $HD_S$ and $HD_T$ are homeomorphic if and only if either case \emph{(2)} occurs or $\nu/ p \equiv \pm \omega / q \mod 1$ \emph{(}that is, $p = q$ and $\nu \equiv \omega \mod p$\emph{)}.
\end{enumerate}
\end{thm}

\begin{thm}
\label{T detailing the correspondence between Heegaard and R-R diagrams of SF spaces over the disk}
The R-R diagrams in Figures~\emph{\ref{PSFFig1a}}, \emph{\ref{PSFFig1b}}, and \emph{\ref{PSFFig1d}} correspond to the Heegaard diagrams $HD_0$, $HD_S$, and $HD_T$ of Theorem~\emph{\ref{SF Heegaard diagram description from BRZ88}} respectively.
\end{thm}

\begin{proof}
First, observe that the curves $\alpha$ in Figures~\ref{PSFFig1a}, \ref{PSFFig1b}, and \ref{PSFFig1d} represent $s^pt^{-q}$, $W_{p,\nu}(u^{-1}, t^q)$, and $W_{q,\omega}(v^{-1}, s^p)$ respectively in $\pi_1(H)$.

Next, consider the diagram of Figure~\ref{PSFFig1a}, and let $C_A$ and $C_B$ be cores of the A-handle and B-handle of $H$. Then $C_A$ and $C_B$ are the exceptional fibers of the Seifert-fibration of $H[\alpha]$. Let $N(C_A)$ and $N(C_B)$ be closed regular neighborhoods of $C_A$ and $C_B$ respectively in $H$, and let $M_A$ and $M_B$ be meridional disks of $N(C_A)$ and $N(C_B)$ respectively.

Observe that the pair of dotted curves $\beta$, $\gamma_s$ on the A-handle of Figure~\ref{PSFFig1a} can be considered to lie on $\partial N(C_A)$, while the pair of dotted curves $\beta'$, $\gamma_t$ on the B-handle of Figure~\ref{PSFFig1a} can be considered to lie on $\partial N(C_B)$. Also observe that as indicated in Remark~\ref{about regular fibers}, the curves $\beta$ and $\beta'$ represent regular fibers of the Seifert-fibration of $H[\alpha]$. Then $\partial M_A$ = $(\beta^{\nu}\gamma_s^p)^{\pm1}$ in $\pi_1(\partial N(C_A))$, while
$\partial M_B$ =  $(\beta'^{\omega}\gamma_t^q)^{\pm1}$ in $\pi_1(\partial N(C_B))$. So $C_A$ and $C_B$ are exceptional fibers of types $\nu / p$ and $\omega / q$ in the Seifert-fibration of $H[\alpha]$.

\begin{figure}[tbp]
\centering
\includegraphics[width = 0.7\textwidth]{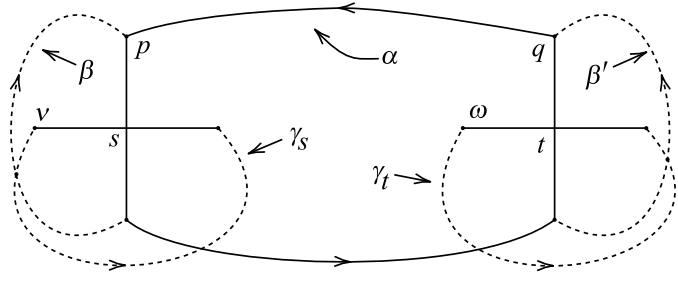}
\caption{Suppose $\nu$, $\omega$, $p$, and $q$ are positive integers such that $0 <\nu < p$, $0 < \omega < q$, $\gcd(\nu , p) = \gcd(\omega, q) = 1$, and $H$ is a genus two handlebody. Then the manifold $H[\alpha]$, obtained by adding a 2-handle to $\partial H$ along a simple closed curve $\alpha$ in $\partial H$ that has an R-R diagram with the form of this figure, is a Seifert-fibered space over $D^2$ with exceptional fibers of types $\nu / p$ and $\omega / q$.}
\label{PSFFig1a}
\end{figure}

\begin{figure}[tbp]
\centering
\includegraphics[width = 0.7\textwidth]{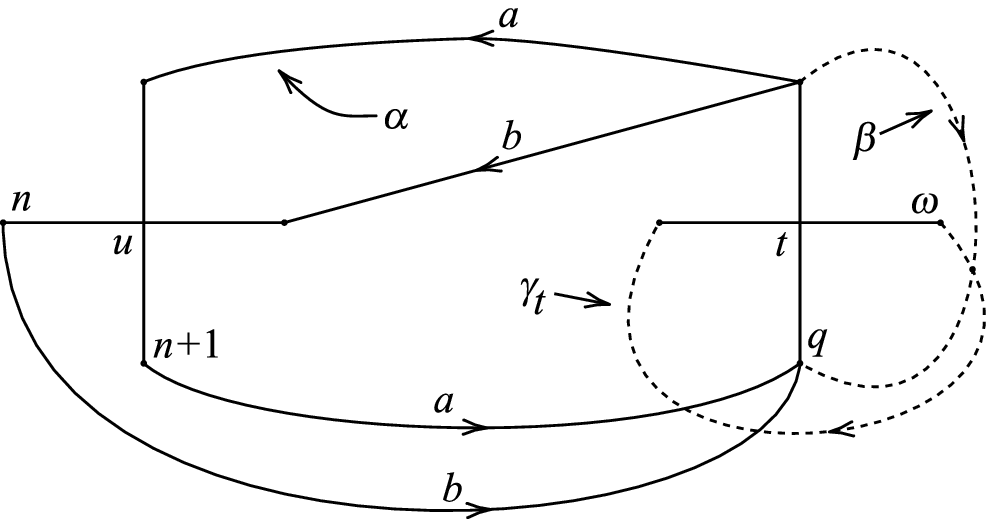}
\caption{Suppose $\nu$, $\omega$, $p$, and $q$ are positive integers such that $1 < \nu < p$, $0 < \omega < q$, and $\gcd(\nu , p) = \gcd(\omega, q) = 1$. In addition, suppose $a$, $b$, and $n$ are positive integers such that $a + b = \nu$, $n\nu + a = p$, and $H$ is a genus two handlebody. Then the manifold $H[\alpha]$, obtained by adding a 2-handle to $\partial H$ along a simple closed curve $\alpha$ in $\partial H$ that has an R-R diagram with the form of this figure, is a Seifert-fibered space over $D^2$ with exceptional fibers of types $\nu / p$ and $\omega / q$.}
\label{PSFFig1b}
\end{figure}

\begin{figure}[tbp]
\centering
\includegraphics[width = 0.7\textwidth]{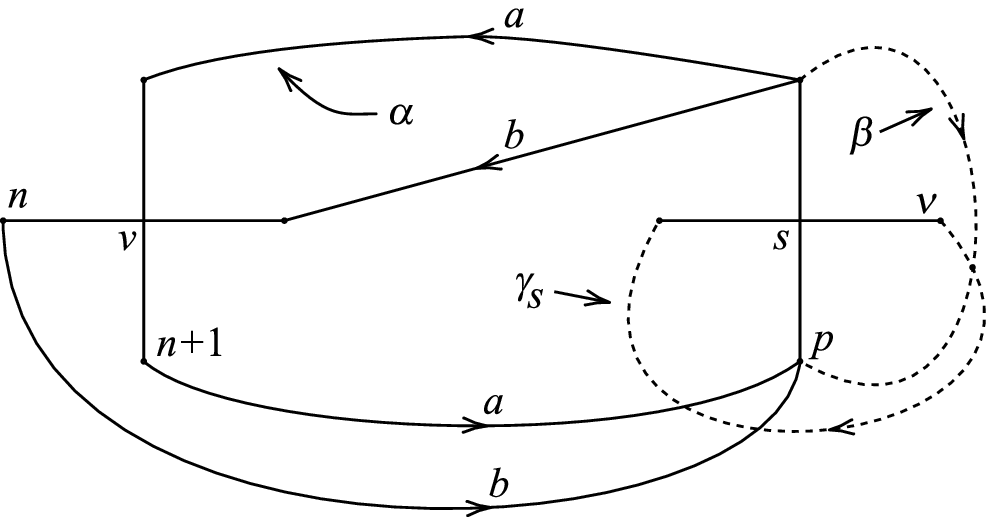}
\caption{Suppose $\nu$, $\omega$, $p$, and $q$ are positive integers such that $0 < \nu < p$, $1 < \omega < q$, and $\gcd(\nu, p) = \gcd(\omega, q) = 1$. In addition, suppose $a$, $b$, and $n$ are positive integers such that $a + b = \omega$, $n\omega + a = q$, and $H$ is a genus two handlebody. Then the manifold $H[\alpha]$, obtained by adding a 2-handle to $\partial H$ along a simple closed curve $\alpha$ in $\partial H$ that has an R-R diagram with the form of this figure, is a Seifert-fibered space over $D^2$ with exceptional fibers of types $\nu / p$ and $\omega / q$.}
\label{PSFFig1d}
\end{figure}

This leaves the diagrams of Figures~\ref{PSFFig1b} and \ref{PSFFig1d}. Since these diagrams are similar, we will only consider Figure~\ref{PSFFig1b} in detail. To start, note that the configuration of the curves $\beta$ and $\gamma_t$ on the B-handle of Figure~\ref{PSFFig1b} is identical to that of $\beta'$ and $\gamma_t$ on the B-handle of Figure~\ref{PSFFig1a}. Since $\beta$ is again a regular fiber in the Seifert-fibration of $H[\alpha]$ when $\alpha$ has an R-R diagram on $\partial H$ with the form of Figure~\ref{PSFFig1b}, the core of the B-handle $C_B$ of $H$ is again an exceptional fiber of type $\omega/ q$. The other exceptional fiber that exists in $H[\alpha]$ when $\alpha$ has an R-R diagram on $\partial H$ with the form of Figure~\ref{PSFFig1b}, arises in a slightly different way.

As in Figure~\ref{PSFFig2c}, two parallel copies of the regular fiber $\beta$ in Figure~\ref{PSFFig1b} bound an essential separating annulus $\mathcal{A}$ in $H$. Cutting $H$ open along $\mathcal{A}$ cuts $H$ into a genus two handlebody $W$ and a solid torus $V$ which has $C_B$ as its core. The curve $\alpha$ lies on $\partial W$, and the R-R diagram of $\alpha$ on $\partial W$ appears in Figure~\ref{PSFFig1c}. Since $\alpha$ is primitive in $W$, $W[\alpha]$ is a solid torus $V'$. Let $C_{V'}$ be the core of $V'$. Then $C_{V'}$ is the second exceptional fiber of the Seifert-fibration of $H[\alpha]$.

Let $M$ be the meridional disk of $V'$, and note that the curves $\beta$ and $\gamma_u$ of Figure~\ref{PSFFig1c} lie on $\partial W$ and $\beta$ and $\gamma_u$ form a basis for $\pi_1(\partial W[\alpha])$. The next step is to obtain an expression for $\partial M$ in $\pi_1(\partial W[\alpha])$ in terms of the basis $\beta$, $\gamma_u$ of $\pi_1(\partial W[\alpha])$.

We can do this by using Lemma~\ref{dot product}. Abelianizing $\pi_1(W)$, we have:
$[\vec{\alpha}]$ = $(-n(a+b)-a, a+b)$ = $(-p, \nu)$, $[\vec{\gamma}_u]$ = $(-1, 0)$, and $[\vec{\beta}\,]$ = $(0, 1)$. By Lemma~\ref{dot product}:
\[
[\vec{\gamma}_u] = \delta ([\vec{\alpha}]^{\perp} \circ [\vec{\gamma}_u]) = \delta((\nu, p)\circ (-1, 0)) = -\delta \nu, \quad \text{and}
\]
\[
[\vec{\beta}\,] = \delta ([\vec{\alpha}]^{\perp} \circ [\vec{\beta}\,]) = \delta((\nu, p)\circ (0, 1)) = \delta p
\]
in $H_1(W[\alpha])$, where $\delta = \pm 1$. It follows that $\partial M$ = $(\beta^{\nu}\gamma_u^p)^{\pm1}$ in $\pi_1(\partial W[R])$. So $C_{V'}$ is an exceptional fiber of type $\nu / p$ in the Seifert-fibration of $H[\alpha]$ when $\alpha$ has an R-R diagram with the form of Figure~\ref{PSFFig1b}.

Similarly, one sees that if $\alpha$ has an R-R diagram with the form of Figure~\ref{PSFFig1d}, then $H[\alpha]$ is also Seifert-fibered over $D^2$ with two exceptional fibers of types $\nu / p$ and $\omega / q$.
\end{proof}

It will be convenient to be able to assume that $n > 1$ in Figure~\ref{PSFFig2a}b. The following lemma shows this can always be done.

\begin{figure}[tbp]
\centering
\includegraphics[width = 0.65\textwidth]{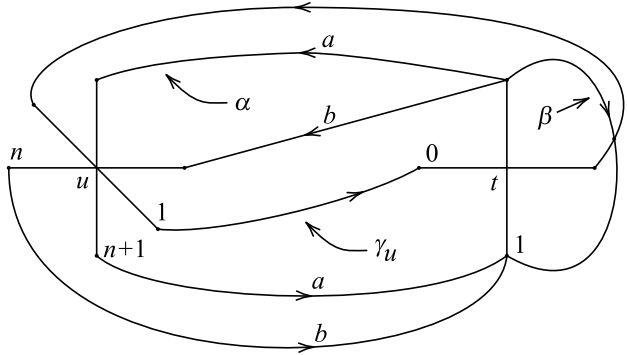}
\caption{Suppose $\alpha$ has an R-R diagram on $\partial H$ with the form of Figure~\ref{PSFFig1b}. Let $W$ be the genus two handlebody obtained when $H$  is cut open along the essential separating annulus in $H$ bounded by two parallel copies of the regular fiber $\beta$ in Figure~\ref{PSFFig1b}. Then $\alpha$ lies on $\partial W$, and $\alpha$ has an R-R diagram on $\partial W$ with the form of this figure. Then $W[\alpha]$ is a solid torus, and the curves $\beta$ and $\gamma_u$ are a basis for $\partial W[\alpha]$. By using Lemma~\ref{dot product} to compute the images of $\beta$ and $\gamma_u$ in $H_1(W[\alpha])$, it is possible to see that the core of $W[\alpha]$ is an exceptional fiber of type $\nu / p$ in the Seifert-fibration of $H[\alpha]$.}
\label{PSFFig1c}
\end{figure}

\begin{figure}[tbp]
\centering
\includegraphics[width = 0.65\textwidth]{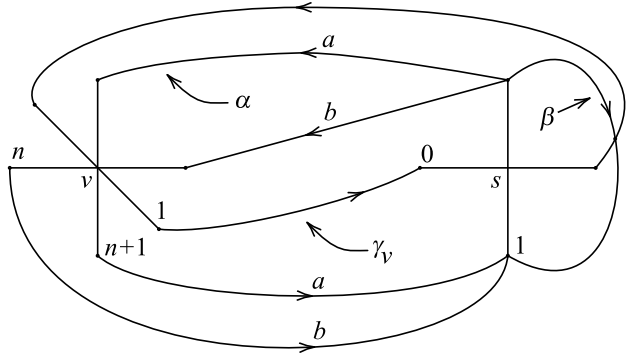}
\caption{Suppose $\alpha$ has an R-R diagram on $\partial H$ with the form of Figure~\ref{PSFFig1d}. Let $W$ be the genus two handlebody obtained when $H$  is cut open along the essential separating annulus in $H$ bounded by two parallel copies of the regular fiber $\beta$ in Figure~\ref{PSFFig1d}. Then $\alpha$ lies on $\partial W$, and $\alpha$ has an R-R diagram on $\partial W$ with the form of this figure. Then $W[\alpha]$ is a solid torus, and the curves $\beta$ and $\gamma_v$ are a basis for $\partial W[\alpha]$. By using Lemma~\ref{dot product} to compute the images of $\beta$ and $\gamma_u$ in $H_1(W[\alpha])$, it is possible to see that the core of $W[\alpha]$ is an exceptional fiber of type $\omega / q$ in the Seifert-fibration of $H[\alpha]$.}
\label{PSFFig1e}
\end{figure}

\begin{lem}
\label{May assume n > 1 in R-R diagrams of SF over the disk}
Suppose $S(\nu/p, \omega/ q)$ has an R-R diagram with the form of Figure~\emph{\ref{PSFFig1b}} or \emph{\ref{PSFFig1d}} with $n = 1$. Then $S(\nu/p, \omega/ q)$ also has another R-R diagram with the form of Figure~\emph{\ref{PSFFig1b}} or \emph{\ref{PSFFig1d}} in which $n > 1$.
\end{lem}

\begin{proof}
Suppose a simple closed curve $\alpha$ in the boundary of a genus two handlebody $H$ has an R-R diagram with the form of Figure~\ref{PSFFig1b} with $n = 1$. For the form of Figure~\ref{PSFFig1d},
the similar argument can apply. Note that since $n=1$, $p=2a+b$ and $\nu=a+b$. It is easy to see that the underlying Heegaard diagram of $\alpha$ on $\partial H$ does not have minimal complexity. Thus in order to have minimal complexity, we can perform a change of cutting disks of $H$, i.e., replace the cutting disk $D_B$ of $H$ which underlies the $B$-handle with a new cutting disk $D'_B$ by bandsumming $D_B$ with the cutting disk $D_A$ in the $A$-handle along the arc of $\alpha$.

Specifically, for the weights $a$ and $b$ in the R-R diagram, since gcd$(a,b)=1$, we can let $b=\rho a+r$, where $\rho\geq0$ and $0\leq r<a$. However we may assume $r>0$, otherwise $a=1$ and
thus $p=2+b$ and $\nu=1+b$, which implies that $\nu\equiv -1 \mod p$ and it follows from Theorem~\ref{SF Heegaard diagram description from BRZ88} that this Heegaard decomposition is homeomorphic to $HD_0$.

Now we record $\alpha$ by starting the $a$ parallel arcs entering into the $-2$-connection in the $A$-handle. It follows from the R-R diagram that $\alpha$ is the product of two subwords $A^{-2}B^q(A^{-1}B^q)^\rho$ and $A^{-2}B^q(A^{-1}B^q)^{\rho+1}$ with $|A^{-2}B^q(A^{-1}B^q)^\rho|=a-r$ and $|A^{-2}B^q(A^{-1}B^q)^{\rho+1}|=r$. We perform a change of cutting disks of the handlebody $H$ which induces an automorphism of $\pi_1(H)$ that takes
$A^{-1}\mapsto A^{-1}B^{-q}$. Then by this change of cutting disks, $A^{-2}B^q(A^{-1}B^q)^\rho$ and $A^{-2}B^q(A^{-1}B^q)^{\rho+1}$ are sent to $A^{-1}B^{-q}(A^{-1})^{\rho+1}$ and $A^{-1}B^{-q}(A^{-1})^{\rho+2}$. Therefore the resulting Heegaard diagram of $\alpha$ has minimal complexity and realizes a new R-R diagram of the form in Figure~\ref{PSFFig1f}, which is the same form as in Figure~\ref{PSFFig1b} with $n>1$.
\end{proof}

\begin{figure}[tbp]
\centering
\includegraphics[width = 0.6\textwidth]{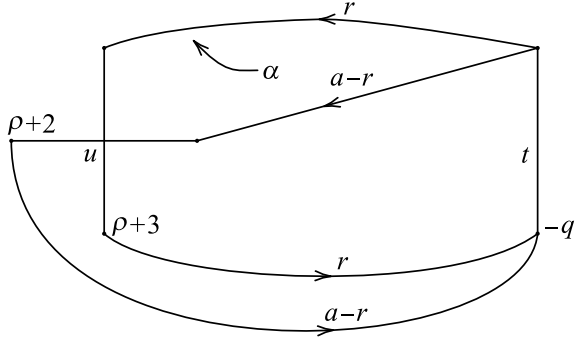}
\caption{A new R-R diagram of $\alpha$ obtained by performing a change of cutting disks of the handlebody $H$ which induces an automorphism of $\pi_1(H)$ that takes
$A^{-1}\mapsto A^{-1}B^{-q}$.}
\label{PSFFig1f}
\end{figure}

\begin{rem}
The following observations are relevant for Lemma~\ref{May assume n > 1 in R-R diagrams of SF over the disk}.
\begin{enumerate}
\item The change of cutting disks of $H$ inducing the automorphism $A^{-1} \mapsto A^{-1}B^{-q}$ in $\pi_1(H)$ corresponds to the change of the cutting disk $D_B$ to a new cutting disk $D_B'$ which is obtained by bandsumming $D_B$ with the cutting disk $D_A$ along the arcs of $\alpha$ $q$ times.
\item The diagram of Figure~\ref{PSFFig1f} corresponds to the Heegaard decomposition of $S((p-\nu)/p, -\omega/q)$, which is homeomorphic to $S(\nu/p, \omega/q)$ as desired.
\end{enumerate}
\end{rem}

\section{R-R diagrams of Seifert-m curves}
\label{Genus two R-R diagrams of Seifert-m curves}

If $\alpha$ is a Seifert-m curve on a genus two handlebody $H$, then by its definition $H[\alpha]$ is a Sefiert-fibered space over the M\"obius band with at most one exceptional fiber. The main result of this section is the following theorem.

\begin{thm}
\label{Thm describing R-R diagrams of SF spaces over the Mobius band}
If $\alpha$ is a nonseparating simple closed curve in the boundary of a genus two handlebody $H$ such that $H[\alpha]$ is a Seifert-fibered space $M$ over the M\"obius band, then $\alpha$ has an R-R diagram of the form shown in Figure~\emph{\ref{SF_on_Mobius1}}, and $\alpha$ represents $AB^sA^{-1}B^s$ in $\pi_1(H)$. There is no loss in taking $s > 0$, and then $M$ has an exceptional fiber if and only if $s > 1$, in which case, $s$ equals the index of the exceptional fiber of $M$.

Conversely, if $\alpha$ has an R-R diagram of the form shown in Figure~\emph{\ref{SF_on_Mobius1}}, then $H[\alpha]$ is Seifert-fibered over the M\"obius band with one exceptional fiber of index $s$ provided that $s>1$. If $s=1$ in Figure~\emph{\ref{SF_on_Mobius1}}, then $H[\alpha]$ is Seifert-fibered space over the M\"obius band with no exceptional fibers.

In addition, a curve $\beta$ shown in Figure~\emph{\ref{SF_on_Mobius1a}}, which is an augmentation of Figure~\emph{\ref{SF_on_Mobius1a}}, is a regular fiber of $H[\alpha]$.
\end{thm}

\begin{figure}[tbp]
\centering
\includegraphics[width = 0.6\textwidth]{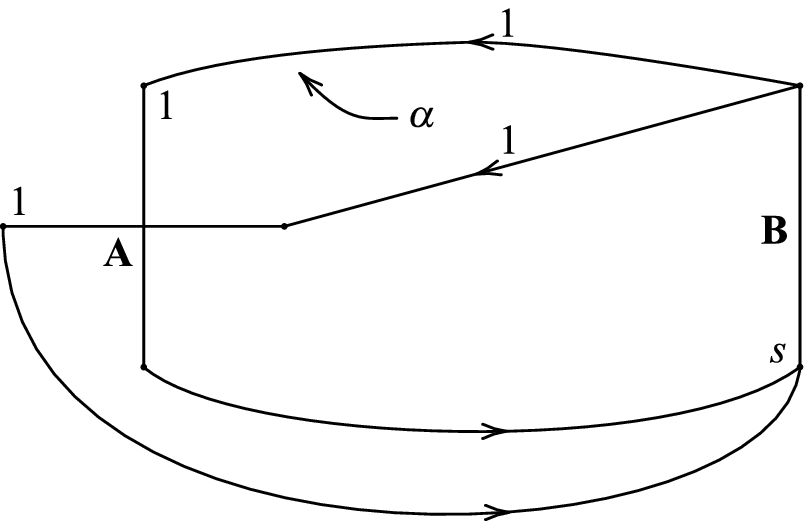}
\caption{If attaching a 2-handle to a genus two handlebody $H$ along a simple closed curve $\alpha$ yields a Seifert-fibered space over the M\"obius band with one exceptional fiber of index $s > 1$, then $\alpha$ has an R-R diagram with the form of this figure in which $\alpha = AB^sA^{-1}B^s$ in $\pi_1(H)$.}
\label{SF_on_Mobius1}
\end{figure}

\begin{proof}
Suppose $H[\alpha]$ is homeomorphic to a Seifert-fibered space $M$ over the M\"obius band. Then $M$ contains an essential nonseparating annulus which is vertical in the Seifert fibration of $M$. (Such an annulus can be easily obtained by starting with a nonseparating arc in the M\"obius band---taking care to choose an arc which misses any exceptional fiber of $M$---and then saturating that arc in the Seifert-fibration of $M$.)

Since $M$ is Seifert-fibered, and not a solid torus, $H[\alpha]$ is $\partial$-irreducible. This implies $\alpha$ intersects every cutting disk of $H$. Now a theorem of Eudave-Mu\~noz applies. It is shown in \cite{EM94} that if $H[\alpha]$ contains an essential nonseparating annulus, then there exists an essential nonseparating annulus $\mathcal{A}$ in $H$, with $\partial \mathcal{A}$ and $\alpha$ disjoint, such that $\mathcal{A}$ is essential in $H[\alpha]$. (Note that in \cite{EM94} the definition of an essential annulus which is properly embedded in a 3-manifold $M$ is that it is incompressible and not $\partial$-parallel.) %(More specifically, see Theorem 1 for an essential annulus in $H$ and Proposition C for the essentiality of $\mathcal{A}$ in $H[\alpha]$, and see the paragraph above the claim 3.7 in page 142 for "nonseparating"in \cite{EM94}.)
Furthermore, it follows from Lemmas 1.10, 1.11 and the argument following them in \cite{H07} that $\mathcal{A}$ is vertical in the Seifert-fibration of $H[\alpha]$.

This suggests looking for all possible R-R diagrams of $\alpha$ by starting with an R-R diagram $\mathcal{D}$ of the boundary components of a nonseparating annulus $\mathcal{A}$ in $\partial H$. Then any R-R diagram of $\alpha$ must be obtained by adding $\alpha$ to $\mathcal{D}$ so that $\alpha$ is disjoint from the curves $\beta$ and $\hat{\beta}$ of $\partial \mathcal{A}$, and $\alpha$ intersects every cutting disk of $H$.

Lemma \ref{R-R diagrams of nonseparating annuli} carries out the first step of this scheme by showing that if $\mathcal{A}$ is a nonseparating essential annulus in a genus two handlebody $H$, and $\beta$ and $\hat{\beta}$ are the components of $\partial \mathcal{A}$ in $\partial H$, then $\beta$ and $\hat{\beta}$ have an R-R diagram $\mathcal{D}$ of the form shown in Figure \ref{SF_on_Mobius5}.

Next, Lemma~\ref{diagrams of curves disjoint from nonseparating annuli} shows that if $\mathcal{D}$ is an R-R diagram with the form of Figure \ref{SF_on_Mobius5}, and a simple closed curve $\alpha$ is added to $\mathcal{D}$ so that $\alpha$ is disjoint from $\beta$ and $\hat{\beta}$ in $\mathcal{D}$, and $\alpha$ intersects every cutting disk of $H$, then the resulting R-R diagram must have the form shown in Figure \ref{SF_on_Mobius2}.

\begin{figure}[tbp]
\centering
\includegraphics[width = 0.7\textwidth]{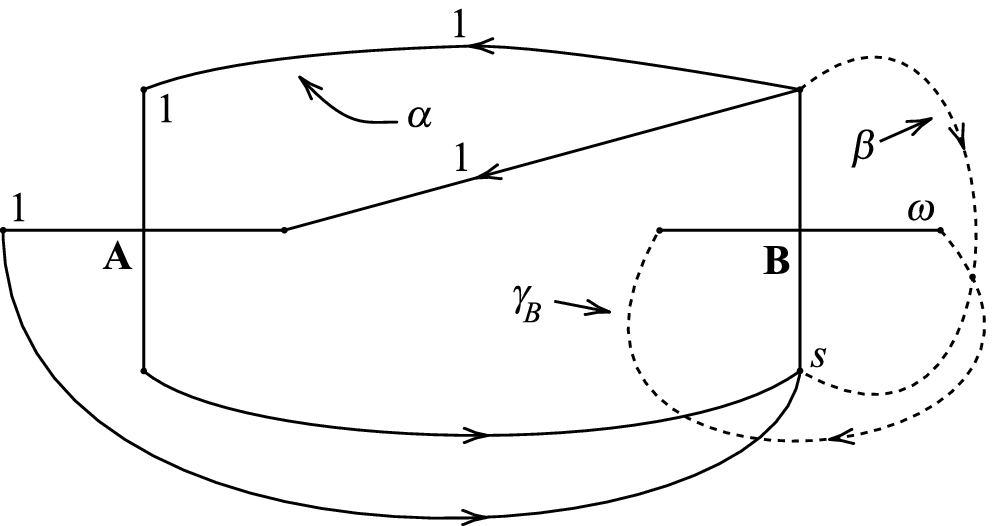}
\caption{Suppose $s$ and $\omega$ are positive integers such that $0 < 2\omega \leq s$, $\gcd(s, \omega) = 1$, and $H$ is a genus two handlebody. Then the manifold $H[\alpha]$, obtained by adding a 2-handle to $\partial H$ along a simple closed curve $\alpha$ in $\partial H$ that has an R-R diagram with the form of this figure, is a Seifert-fibered space over the M\"obius band with one exceptional fiber of type $\omega / s$, whose regular fiber is the curve $\beta$.}
\label{SF_on_Mobius1a}
\end{figure}

Lemmas~\ref{must have (a,b) = (0,1) when s > 1} and \ref{must have (a,b) = (0,1) when s = 1} finish the argument by showing that if $H[\alpha]$ is Seifert-fibered over the M\"obius band, then $(a,b) = (0,1)$ in Figure~\ref{SF_on_Mobius2}, so Figure~\ref{SF_on_Mobius2} reduces to Figure~\ref{SF_on_Mobius1}. In addition, these lemmas show that $H[\alpha]$ has an exceptional fiber if and only if $s > 1$, and when $s > 1$, $s$ equals the index of the exceptional fiber of $H[\alpha]$.

Now we prove the second statement of the theorem. First, suppose $\alpha$ and $\beta$ have an R-R diagram of the form shown in Figure~\ref{SF_on_Mobius1a} with $s$ replaced by $1$, which is an augmentation of the R-R diagram of $\alpha$ in Figure~\ref{SF_on_Mobius1}. We will show that
$H[\alpha]$ is a Seifert-fibered space over the M\"obius band with no exceptional fibers whose regular fiber is represented by the curve $\beta$. Consider a properly embedded nonseparating annulus $\mathcal{A}$ in $H$ whose boundary consists of the curves $\beta$ and $\hat{\beta}$, where $\hat{\beta}$ is a curve illustrated in Figure~\ref{SF_on_Mobius5}.
Figure~\ref{SF_on_Mobius1b}a shows the genus two handlebody $H$, the simple closed curve $\alpha$ on $\partial H$, and the annulus $\mathcal{A}$ with $\partial \mathcal{A}=\beta\cup \hat{\beta}$ oriented, which realize the R-R diagram of $\alpha$ and $\beta$ in Figure~\ref{SF_on_Mobius1a} with $s=1$ and the R-R diagram of $\hat{\beta}$ in Figure~\ref{SF_on_Mobius5}. Let $H/\mathcal{A}$ be
the manifold obtained by cutting $H$ along $\mathcal{A}$. Let $\mathcal{A}_1$ and
$\mathcal{A}_2$, $\beta_1$ and $\beta_2$, and $\hat{\beta}_1$ and $\hat{\beta}_1$ be the copies of $\mathcal{A}$, $\beta$, and $\hat{\beta}$ in $H/\mathcal{A}$. Then it is easy to see that $H/\mathcal{A}$ is
a genus two handlebody such that $\mathcal{A}_1$ and
$\mathcal{A}_2$ together with $\partial{\mathcal{A}_1}=\beta_1\cup\hat{\beta}_1$ and $\partial{\mathcal{A}_2}=\beta_2\cup\hat{\beta}_2$ lie in
$\partial (H/\mathcal{A})$ as shown in Figure~\ref{SF_on_Mobius1b}b.

\begin{figure}[tbp]
\centering
\includegraphics[width = 0.65\textwidth]{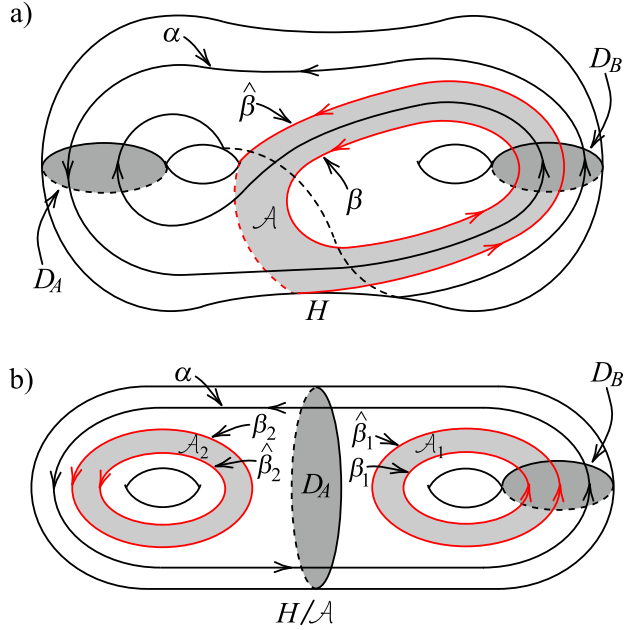}
\caption{The genus two handlebody $H$, the simple closed curve $\alpha$ on $\partial H$, and the annulus $\mathcal{A}$ with $\partial \mathcal{A}=\beta\cup \hat{\beta}$ oriented, which realize the R-R diagram of $\alpha$ in Figure~\ref{SF_on_Mobius1} with $s=1$ and the R-R diagram of $\beta$ and $\hat{\beta}$ in Figure~\ref{SF_on_Mobius5} in a), and the manifold $H/\mathcal{A}$ obtained by cutting $H$ along $\mathcal{A}$ in b), which is a genus two handlebody.}
\label{SF_on_Mobius1b}
\end{figure}

Since $\alpha$ is disjoint from $\mathcal{A}$, $H[\alpha]$ is obtained from $(H/\mathcal{A})[\alpha]$ by gluing the two copies $\mathcal{A}_1$ and $\mathcal{A}_2$
such that the orientations of their boundaries $\beta_1\cup\hat{\beta}_1$ and $\beta_2\cup\hat{\beta}_2$ match. However, we can observe from Figure~\ref{SF_on_Mobius1b}b that $\alpha$ is primitive in the genus two handlebody $H/\mathcal{A}$ and thus $(H/\mathcal{A})[\alpha]$ is a solid torus. Thus gluing $\mathcal{A}_1$ and $\mathcal{A}_2$ in the boundary of the solid
torus $(H/\mathcal{A})[\alpha]$ yields Seifert-fibered over the M\"obius band with no exceptional fibers such that $\beta$ is a regular fiber. Therefore $H[\alpha]$ is a Seifert-fibered space over the M\"obius band with no exceptional fibers whose regular fiber is represented by the curve $\beta$.

Now we suppose that $\alpha$ and $\beta$ have an R-R diagram of the form shown in Figure~\ref{SF_on_Mobius1a} with $s>1$. Similarly as in the proof of Theorem~\ref{Thm describing R-R diagrams of SF spaces over the disk}, the two parallel copies of $\beta$ bound an essential separating annulus $\mathcal{A}'$ in $H$ as shown in Figure~\ref{PSFFig2c}, which cuts $H$ apart into a genus two handlebody $W$ and a solid torus $V$. Note that $\alpha$ lies in the boundary of $W$. To complete the proof, it suffices to show that $W[\alpha]$ is a Seifert-fibered space over the M\"obius band with no exceptional fibers and the curve $\beta$ is a regular fiber of $W[\alpha]$.

A component of $D_B\cap W$ is a cutting disk $D_C$ of $W$ such that $\alpha$ and $\beta$ intersect $D_C$ transversely once. This implies that the R-R diagram of $\alpha$ and $\beta$ with respect to $\{D_A, D_C\}$ has the form in Figure~\ref{SF_on_Mobius1a} with $s$ replaced by $1$. Therefore by the argument above in the case that $s=1$, we see that $W[\alpha]$ is a Seifert-fibered space over the M\"obius band with no exceptional fibers and the curve $\beta$ is a regular fiber of $W[\alpha]$, as desired.
\end{proof}

\begin{rem}
In Theorem~\ref{Thm describing R-R diagrams of SF spaces over the Mobius band}, if $s=1$, then $H[\alpha]$ is a Seifert-fibered space $M$ over the M\"obius band with no exceptional fibers. Then $M$ admits another Seifert-fibered structure, i.e., $M$ is a Seifert-fibered space over $D^2$ with two exceptional fibers of both indexes 2. Therefore we can regard a Seifert-m curve $\alpha$ in this case as a Seifert-d curve, and thus we may assume $s>1$ in the R-R diagram of a Seifert-m curve $\alpha$.

(2) Figure~\ref{SF_on_Mobius1a} shows an R-R diagram of $\alpha$ such that $H[\alpha]$ is a Seifert-fibered space over the M\"obius band with one exceptional fiber of type $\omega / s$, whose regular fiber is the curve $\beta$.
\end{rem}

\begin{lem}
\label{R-R diagrams of nonseparating annuli}
Suppose $\mathcal{A}$ is an essential nonseparating annulus properly embedded in a genus two handlebody $H$. Let $\beta$ and $\hat{\beta}$ be the two curves in $\partial H$ that form $\partial \mathcal{A}$. Then the pair $\beta$, $\hat{\beta}$ have an R-R diagram that appears in Figure~\emph{\ref{SF_on_Mobius5}}.
\end{lem}

\begin{figure}[tbp]
\centering
\includegraphics[width = 0.6\textwidth]{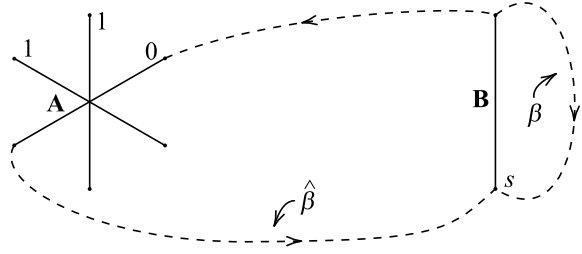}
\caption{If $\mathcal{A}$ is a nonseparating essential annulus in a genus two handlebody $H$ with $\partial \mathcal{A}$ = $\beta \cup \hat{\beta}$, then there exists $s > 0$ such that $\beta$ and $\hat{\beta}$ have an R-R diagram with the form of this figure.}
\label{SF_on_Mobius5}
\end{figure}

\begin{proof}
Given $\mathcal{A}$ and its boundary components $\beta$ and $\hat{\beta}$, we claim the following.

\begin{claim}\label{disjoint and spanning arcs}
There exists a complete set of cutting disks $\{D_A, D_B\}$ of $H$ such that one of the cutting disks, say $D_A$, is disjoint from $\mathcal{A}$ and $\mathcal{A} \cap D_B$ consists of a set of $s > 0$ essential spanning arcs in $\mathcal{A}$.
\end{claim}

\begin{figure}[tbp]
\centering
\includegraphics[width = 0.5\textwidth]{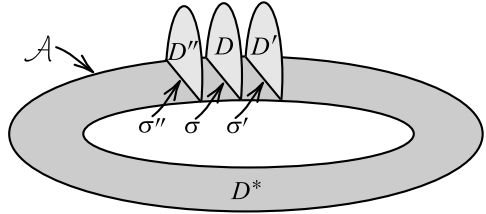}
\caption{Gluing two copies $D'$ and $D''$ of $D$ to the disk $D^*$ of $\mathcal{A}$
along $\sigma'$ and $\sigma''$ yields a disk $D_C$.}
\label{SF_on_Mobius5b}
\end{figure}
\begin{proof}
First, we show that there exists a cutting disk of $H$ disjoint from $\mathcal{A}$. Let $D_A$ be a cutting disk of $H$ which intersects $\mathcal{A}$ minimally.

Suppose $D_A\cap \mathcal{A}\neq \varnothing$. Then we may assume that $D_A$ intersects $\mathcal{A}$ essentially and $D_A\cap \mathcal{A}$ consists of properly embedded disjoint arcs and disjoint circles. However by the incompressibility of $\mathcal{A}$, irreducibility of $H$, and the minimality condition rule out circle intersections.
Suppose $\gamma$ is an outermost arc of $D_A\cap \mathcal{A}$ which cuts a disk $D$ of $\mathcal{A}$. Then $\gamma$ also cuts $D_A$ into two subdisks $D_1$ and $D_2$ of $D_A$. Consider two disks $D\cup_\gamma D_1$ and $D\cup_\gamma D_2$, which are obtained by gluing $D$ and $D_1$, and $D$ and $D_2$ respectively along $\gamma$. Then since
$D_A$ is a cutting disk which means that it is nonseparating, at least one of $D\cup_\gamma D_1$ and $D\cup_\gamma D_2$ is nonseparting and thus is a cutting disk. But this cutting disk intersects $\mathcal{A}$ less than $D_A$. This is a contradiction to the minimality.

Suppose $\sigma$ is a spanning arc of $D_A\cap \mathcal{A}$ in $\mathcal{A}$ which is outermost in $D_A$. Let $D'$ and $D''$ be two copies of $D$, and let $\sigma'=D'\cap \mathcal{A}$ and $\sigma''=D''\cap \mathcal{A}$
as shown in Figure~\ref{SF_on_Mobius5b}. Also let $D^*$ be a disk $\mathcal{A}-N(\sigma)$ such that $\sigma'$ and $\sigma''$ are included in $\partial D^*$.
Gluing the two copies $D'$ and $D''$ to the disk $D^*$ along $\sigma'$ and $\sigma''$ yields a disk $D_C$. Since $\mathcal{A}$ is nonseparating, from the construction $D_C$ is nonseparating (and thus is a cutting disk) and also does not intersect $\mathcal{A}$, a contradiction. Therefore there exists a cutting disk $D_A$ of $H$ disjoint from $\mathcal{A}$.

Now let $D_B$ be a cutting disk of $H$ chosen so that $\{D_A, D_B\}$ is a complete set of cutting disks of $H$ and $D_B$ intersects $\mathcal{A}$ minimally. By applying the same argument above, we can show that $D_B\cap \mathcal{A}$ consists of spanning arcs of $\mathcal{A}$. Note that $|D_B\cap \mathcal{A}|>0$, otherwise $\mathcal{A}$ would embed properly in a 3-ball and thus not be essential. This completes the proof of the claim.
\end{proof}

By the claim, there exists a complete set of cutting disks $\{D_A, D_B\}$ of $H$ such that $D_A$ is disjoint from $\mathcal{A}$ and $\mathcal{A} \cap D_B$ consists of a set of $s > 0$ essential spanning arcs in $\mathcal{A}$. Now consider the solid torus $V$ obtained by cutting $H$ open along $D_A$. Then $\partial V$ contains two disks $D_A^+$ and $D_A^-$, which are copies of $D_A$.
The simple closed curves $\beta$ and $\hat{\beta}$ also lie in $\partial V$ and cut $\partial V$ into two annuli, say $A^+$ and $A^-$. Since $\beta$ and $\hat{\beta}$ are not isotopic in $\partial H$, $D_A^+$ lies in the interior of one of these annuli, and $D_A^-$ lies in the interior of the other.
It is easy to see that there exists an arc $\tau$ in $\partial V$ connecting $D_A^+$ and $D_A^-$ which $|\tau \cap (\beta \cup \hat{\beta})| = 1$. Furthermore
since a regular neighborhood of $D_A^+\cup D_A^-\cup \tau$ is a disk in $\partial V$, we can isotope $D_A^+\cup D_A^-\cup \tau$ keeping $|\tau \cap (\beta \cup \hat{\beta})| = 1$
so that $D_A^+\cup D_A^-\cup \tau$ is disjoint from the meridional disk $D_B$ of $V$. Since $|\tau \cap (\beta \cup \hat{\beta})| = 1$, now we may assume that $|\tau \cap\hat{\beta}| = 1$.

Next, let $T_A$ be a once-punctured torus in $\partial H$ which is a regular neighborhood in $\partial H$ of $\tau \cup \partial D_A$, with $T_A$ chosen so that $T_A$ is disjoint from $\partial D_B$, and so that $T_A$ intersects $\beta \cup \hat{\beta}$ minimally. Then, if $\mathcal{D}$ is the R-R diagram of $\beta$ and $\hat{\beta}$, whose A-handle corresponds to $T_A$, then $\hat{\beta}$ crosses the A-handle of $\mathcal{D}$ in a single connection without intersecting $D_A$, while $\beta$ lies completely in the B-handle of $\mathcal{D}$. So $\mathcal{D}$ has the form of the diagram in Figure~\ref{SF_on_Mobius5}.
\end{proof}

\begin{lem}
\label{diagrams of curves disjoint from nonseparating annuli}
If $\mathcal{D}$ is an R-R diagram with the form of Figure \emph{\ref{SF_on_Mobius5}}, and a simple closed curve $\alpha$ is added to $\mathcal{D}$ so that $\alpha$ is disjoint from $\beta$ and $\hat{\beta}$ in $\mathcal{D}$, and $\alpha$ intersects every cutting disk of $H$, then the resulting R-R diagram must have the form shown in Figure \emph{\ref{SF_on_Mobius2}}.
\end{lem}

\begin{figure}[tbp]
\centering
\includegraphics[width = 0.65\textwidth]{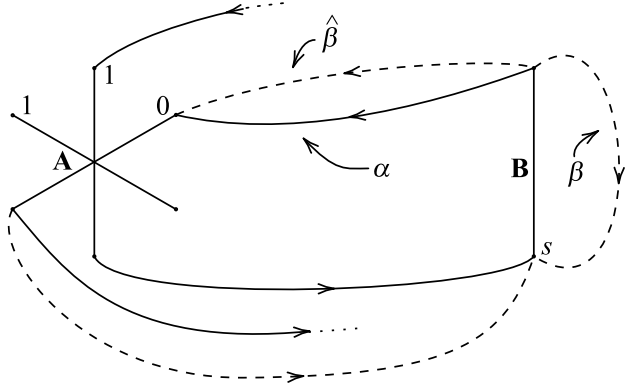}
\caption{This figure shows that a simple closed curve $\alpha$ on $\partial H,$ which is disjoint from the curves $\beta$ and $\hat{\beta}$ forming $\partial \mathcal{A}$, can not contain both a 0-connection and a 1-connection on the A-handle of this diagram.  (Otherwise $\alpha$ is forced to spiral endlessly.)}
\label{SF_on_Mobius4}
\end{figure}

\begin{figure}[tbp]
\centering
\includegraphics[width = 0.75\textwidth]{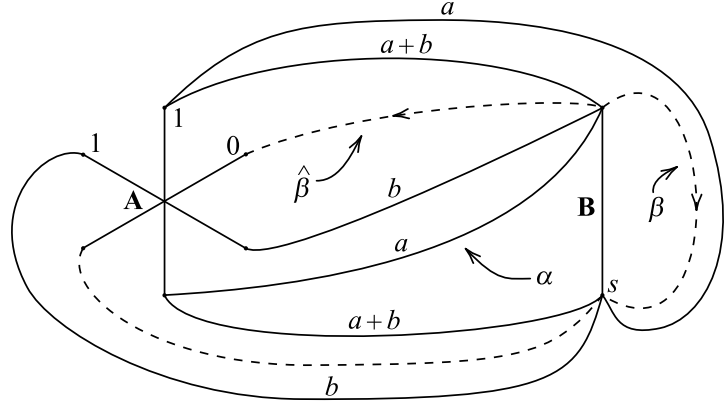}
\caption{If a 2-handle is added to the genus two handlebody $H$ of Figure~\ref{SF_on_Mobius5} along a simple closed curve $\alpha$, disjoint from the curves $\beta$ and $\hat{\beta}$ forming $\partial \mathcal{A}$, the annulus $\mathcal{A}$ of Figure~\ref{SF_on_Mobius5} is essential in $H[\alpha]$, and $\alpha$ intersects every cutting disk of $H$, then there are nonnegative integers $a$ and $b$ such that $\alpha$, $\beta$ and $\hat{\beta}$ have an R-R diagram with the form of this figure.}
\label{SF_on_Mobius2}
\end{figure}

\begin{proof}
Since $\alpha$ must intersect every cutting disk of $H$, it must traverse both the A-handle and the B-handle of $\mathcal{D}$. In particular, $\alpha$ must have 1-connections on the A-handle of $\mathcal{D}$ and s-connections on the B-handle of $\mathcal{D}$. Now Figure~\ref{SF_on_Mobius4} shows that $\alpha$ can't have both 1-connections and 0-connections on the A-handle of $\mathcal{D}$, otherwise $\alpha$ is forced to spiral endlessly. So $\alpha$ has only 1-connections on the A-handle of $\mathcal{D}$. Now it is not hard to see that there must exist nonnegative integers $a$ and $b$ such that the diagram of $\alpha$, $\beta$, and $\hat{\beta}$ has the form of Figure~\ref{SF_on_Mobius2}.
\end{proof}

\begin{lem}
\label{must have (a,b) = (0,1) when s > 1}
Suppose $\alpha$ has an R-R diagram with the form of Figure~\emph{\ref{SF_on_Mobius2}}, $H[\alpha]$ is Seifert-fibered over the M\"obius band, and $s > 1$ in Figure~\emph{\ref{SF_on_Mobius2}}. Then $(a,b) = (0,1)$, Figure~\emph{\ref{SF_on_Mobius2}} reduces to Figure~\emph{\ref{SF_on_Mobius1}}, and $\alpha = AB^sA^{-1}B^s$ in $\pi_1(H)$. Furthermore, the core of the B-handle of $H$ is an exceptional fiber of index $s$ in the Seifert-fibration of $H[\alpha]$.
\end{lem}

\begin{proof}
Since the annulus $\mathcal{A}$ is vertical in the Seifert-fibration of $H[\alpha]$, its boundary components $\beta$ and $\hat{\beta}$ are regular fibers. Then as in Figure~\ref{PSFFig2c} two parallel copies of $\beta$ bound an essential separating annulus $\mathcal{A}'$ in $H$ such that $\mathcal{A}'$ is saturated and vertical in the Seifert-fibration of $H[\alpha]$, and $\mathcal{A}'$ cuts $H$ into a solid torus $V$ and a genus two handlebody $W$, with $\alpha$ lying on $\partial W$.

Let $\lambda$ be the core of $V$. (Note that $\lambda$ is also a core of the B-handle of $H$.) Then, since $s > 1$, $\lambda$ is an exceptional fiber of index $s$ in the Seifert-fibration of $H[\alpha]$. And, since the Seifert-fibration of $H[\alpha]$ can have at most one exceptional fiber, $\lambda$ is the only exceptional fiber in $H[\alpha]$.
It follows that the manifold $W[\alpha]$ is Seifert-fibered over the M\"obius band with no exceptional fibers. Then, using well-known formulas for presentations of Seifert-fibered spaces, one gets $\pi_1(W[\alpha])$ = $\langle x, y \,|\,x^2y^2 \rangle$.

By a result of Zieschang in \cite{Z77}, $\pi_1(W[\alpha])$ has only one Nielsen equivalence class of generators. It follows that if $\langle x, y \,|\,\mathcal{R} \rangle$ is a one-relator  presentation of $\pi_1(W[\alpha])$, then there is an automorphism of the free group $F(x,y)$
which carries $\mathcal{R}$ onto a cyclic conjugate of $x^2y^2$ or its inverse.

Now it is not hard to see that one obtains a one-relator presentation of $\pi_1(W[\alpha])$ from the one-relator presentation $\langle A, B \,|\,\alpha \rangle$ of $\pi_1(H[\alpha])$ by setting $s = 1$ in Figure~\ref{SF_on_Mobius2}. But if $s = 1$ in
Figure~\ref{SF_on_Mobius2}, then the Heegaard diagram underlying the R-R diagram in Figure~\ref{SF_on_Mobius2} has a graph with the form of Figure~\ref{SF_on_Mobius3}. This graph shows $\alpha$ has minimal length under automorphisms of $F(A,B)$. It follows that if $\alpha$ is a cyclic conjugate of $A^2B^2$ or its inverse in $F(A,B)$, then $a + b = 1$. So $(a,b) = (1,0)$ or $(a,b) = (0,1)$. If $(a,b) = (1,0)$, then $\alpha = ABA^{-1}B^{-1}$, which is a commutator, $\alpha$ separates $\partial H$, and $\alpha$ is not an automorph of $(A^2B^2)^{\pm1}$ in $F(A,B)$. The only remaining possibility is $(a,b) = (0,1)$. In this case, Figure~\ref{SF_on_Mobius2} reduces to Figure~\ref{SF_on_Mobius1}, and $\alpha = AB^sA^{-1}B^s$ in $\pi_1(H)$, as desired.
\end{proof}

\begin{lem}
\label{must have (a,b) = (0,1) when s = 1}
Suppose $\alpha$ has an R-R diagram with the form of Figure~\emph{\ref{SF_on_Mobius2}}, $H[\alpha]$ is Seifert-fibered over the M\"obius band, and $s = 1$ in Figure~\emph{\ref{SF_on_Mobius2}}. Then $(a,b) = (0,1)$, Figure~\emph{\ref{SF_on_Mobius2}} reduces to Figure~\emph{\ref{SF_on_Mobius1}}, and $\alpha = ABA^{-1}B$ in $\pi_1(H)$. Furthermore, the Seifert-fibration of $H[\alpha]$ has no exceptional fibers.
\end{lem}

\begin{proof}
As in Lemma~\ref{must have (a,b) = (0,1) when s > 1}, the annulus $\mathcal{A}$ is vertical in the Seifert-fibration of $H[\alpha]$, and its boundary components $\beta$ and $\hat{\beta}$ are regular fibers. Then two parallel copies of $\beta$ bound a separating annulus $\mathcal{A}'$ in $H$ such that $\mathcal{A}'$ is saturated and vertical in the Seifert-fibration of $H[\alpha]$, and $\mathcal{A}'$ cuts $H$ into a solid torus $V$ and a genus two handlebody $W$, with $\alpha$ lying on $\partial W$.
Let $\lambda$ be the core of $V$. (Note that like the $s > 1$ case of Lemma~\ref{must have (a,b) = (0,1) when s > 1}, $\lambda$ is also a core of the B-handle of $H$, but unlike the $s > 1$ case, here $\mathcal{A}'$ is parallel into $\partial H$.)
Then, since $s = 1$, $\lambda$ is a regular fiber in the Seifert-fibration of $H[\alpha]$. By performing Dehn surgery on $\lambda$, one can change $H$ into another genus two handlebody $H'$ such that $H'[\alpha]$ is Seifert-fibered over the M\"obius band with $\lambda$ as an exceptional fiber of index $s' > 1$. And then, since $H'[\alpha]$ is Seifert-fibered over the M\"obius band, $\lambda$ must be its only exceptional fiber. This implies that the Seifert-fibration of $H[\alpha]$ had no exceptional fibers.

To finish, notice that the R-R diagram of $\alpha$ on $\partial H'$ is obtained from the diagram of $\alpha$ on $\partial H$ by replacing $s$ in Figure~\ref{SF_on_Mobius2} with $s'$. Now the argument used in Lemma~\ref{must have (a,b) = (0,1) when s > 1} applies and shows that, as before, $(a,b) = (0,1)$ and the diagram of Figure~\ref{SF_on_Mobius2} reduces to that of Figure~\ref{SF_on_Mobius1}.
\end{proof}

\begin{figure}[tbp]
\centering
\includegraphics[width = 0.5\textwidth]{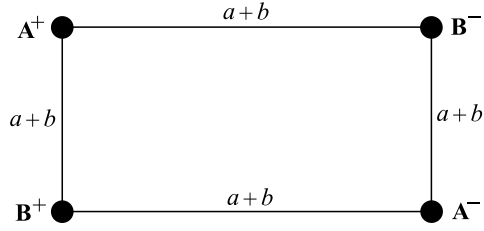}
\caption{If $s = 1$ in Figure~\ref{SF_on_Mobius2}, the Heegaard diagram under\-lying the R-R diagram in Figure~\ref{SF_on_Mobius2} has a graph with this form.}
\label{SF_on_Mobius3}
\end{figure}

%\clearpage

\end{document}